\documentclass[english, 12pt]{amsart}

\usepackage[english]{isodate}
\usepackage{tikz}
\usetikzlibrary{cd}
\usepackage{aliascnt}
\usepackage{mathrsfs}
\usepackage[all,poly,knot]{xy}
\usepackage{hyperref}
\usepackage{amssymb,amsbsy,amsmath,amsfonts,amssymb,amscd,
	graphics,color,footmisc,fancyhdr,multicol,fancybox,
	graphicx,mathrsfs,rotating,ifthen,wasysym}

\usepackage{lipsum}

\usepackage{times}
\usepackage{mathtools}
\usepackage{pdfpages}
\usepackage{multirow}
\usepackage{nccrules}
\usepackage{textcomp}
\usepackage{comment}
\usepackage{framed}
\usepackage[all]{xy}
\usepackage{graphicx}
\usepackage{pgf,tikz}
\usepackage{mathrsfs}
\usetikzlibrary{arrows}
\usepackage{lipsum}
\usepackage{mathtools}

\DeclareMathOperator{\dif}{\text{\normalfont d}}

\def\log{\mathrm{log}\,}

\theoremstyle{plain}

\newtheorem{thm}{Theorem}[section]

\newtheorem{obs}[thm]{{Observation}} 

\newtheorem{lem}[thm]{{Lemma}}

\theoremstyle{remark}
\newtheorem{rmk}[thm]{Remark}

\numberwithin{equation}{section}

\newtheorem{ques}[thm]{{\bf Question}}

\usepackage[top=2cm, bottom=2cm, left=2.5cm, right=2.5cm]{geometry}

\usepackage[hyperpageref]{backref}

\usepackage{xcolor}
\hypersetup{
colorlinks,
linkcolor={red!50!black},
citecolor={blue!62!black},
urlcolor={blue!80!black}
}
\theoremstyle{plain}
\newcommand{\thistheoremname}{}
\newtheorem*{genericthm*}{\thistheoremname}
\newenvironment{namedthm*}[1]{\renewcommand{\thistheoremname}{#1}%
\begin{genericthm*}}
{\end{genericthm*}}

\newtheoremstyle{named}{}{}{\itshape}{}{\bfseries}{.}{.5em}{\thmnote{#3's }#1}
\theoremstyle{named}

\makeatletter
\newcommand\thankssymb[1]{\textsuperscript{\@fnsymbol{#1}}}
\makeatother
\begin{document} 
\title{
	Universal holomorphic maps with slow growth
\\
I. An Algorithm
}


\author{Bin Guo}

\address{Academy of Mathematics and Systems Sciences, Chinese Academy of Sciences, Beijing 100190, China; School of Mathematical Sciences, University of Chinese Academy of Sciences, Beijing
100049, China}
\email{guobin181@mails.ucas.ac.cn}

\author{Song-Yan Xie }
\address{Academy of Mathematics and System Science \& Hua Loo-Keng Key Laboratory
	of Mathematics, Chinese Academy of Sciences, Beijing 100190, China; 
	School of Mathematical Sciences, University of Chinese Academy of Sciences, Beijing
100049, China}
\email{xiesongyan@amss.ac.cn}

\date{\today}


\begin{abstract}
	We design an {\sl Algorithm} to  fabricate universal holomorphic    maps  between any two complex Euclidean spaces, within preassigned transcendental growth rate. 
	As by-products,  universal holomorphic maps  from $\mathbb{C}^n$ to $\mathbb{CP}^m$ ($n\leqslant m$) and to complex tori having slow growth are obtained. We take inspiration  from Oka manifolds theory, Nevanlinna theory, and hypercyclic operators theory.
\end{abstract}

\maketitle

\hfill
{\em dedicated to the memory of Nessim Sibony}

\bigskip

\section{\bf Introduction}
About one century ago, in the  space $\mathcal{H}(\mathbb{C})$ 
of entire holomorphic functions, 
 Birkhoff~\cite{Birkhoff1929}
constructed surprising examples 
$F: 
\mathbb{C}\rightarrow \mathbb{C}$ 
with the  {\sl universal} property that every  $H
\in \mathcal{H}(\mathbb{C})$ can be reached as a 
limit by translations of $F$, precisely,
for certain sequence $\{a_{H, i}\}_{i\geqslant 1}$
of complex numbers there holds
\begin{equation}
	\label{Birkhoff example}
	\lim_{i\rightarrow \infty}
	F(z+a_{H, i})=H(z)
	\qquad
	{\scriptstyle (\forall\, z\,\in\, \mathbb{C})}.
\end{equation}
 The original construction often appears in exercise of complex analysis (e.g.~\cite[pp.~69--70, ex.~5]{Stein-book-complex-analysis}). 

For source spaces with certain symmetry, there might be also  ``universal'' holomorphic maps into a certain class of complex manifolds. 
In hindsight,
the  existence of such  maps
is intimately related to Oka theory~\cite{Forstneric-Oka-book,  MR4547869}, which originated (1939) in Kiyoshi Oka's celebrated articles on Cousin Problems I and II~\cite{MR3309925}. Major 
developments have been made by Hans Grauert~\cite{MR98199}, 
 Mikhail Gromov~\cite{MR1001851} 
 and Franc Forstneri\v{c}~\cite{MR2199229, MR2554568} (cf.~\cite{MR4547869} for detailed account). 
  The simplest, among a dozen, characterization of Oka manifolds is given by
Runge's approximation property~\cite{MR2199229}.

\bigskip\noindent {\bf Definition.} A complex manifold $Y$ is an Oka manifold if every holomorphic map 
from a neighbourhood of a compact convex set $K$ in a Euclidean space $\mathbb{C}^n$ (for any $n\in \mathbb{Z}_+$) to
$Y$ is a uniform limit on $K$ of entire maps $\mathbb{C}^n \rightarrow Y$.

\bigskip

Basic examples of Oka manifolds include  $\mathbb{C}^m$,  $\mathbb{CP}^m$, and complex tori $\mathbb{T}^m$. A result of Kusakabe~\cite[Theorem 1.4]{Kusababe-first-paper} shows that, given an Oka manifold $Y$, for any source space $X$ with mild symmetry,  there are  plenty of
``universal'' holomorphic maps from $X$ to $Y$. 

One  philosophy 
in algebraic geometry is that  key  information of a variety can be read off from  moduli spaces of certain curves or subvarieties it contains. 
Similarly, for seeking quantitative invariants of a given Oka manifold $Y$, we shall likewise study 
 the spaces $\mathsf{Hol}(X, Y)$ of holomorphic maps from various open source spaces $X$ to $Y$,
and
certain subspaces $\mathscr{U}(X, Y)$  consisting of ``universal'' holomorphic maps
are particularly interesting.

There is a natural {\sl Nevanlinna functional} on  the space $\mathsf{Hol}(X, Y)$, which associates an element $F: X\rightarrow Y$ with its (certain) Nevanlinna characteristic function $T_F(\bullet): \mathbb{R}_{\geqslant 1}\rightarrow \mathbb{R}_{\geqslant 0}$, capturing the asymptotic growth rate of $F$.  An interesting problem is thus studying   ``possible minimum'' of  the Nevanlinna functional on subspaces  $\mathscr{U}(X, Y)\subset \mathsf{Hol}(X, Y)$ for  Oka manifolds $Y$ and various  source spaces $X$.  One  potential application is about classification of Oka manifolds.

We start with the fundamental case where  $X=\mathbb{C}^n$ and $Y=\mathbb{C}^m,  \mathbb{CP}^m$, or $\mathbb{T}^m$, having certain linear structure. 
One motivation is a challenge~\cite[Problem 9.1]{Dinh-Sibony-list} raised by Dinh and Sibony,
about determining  minimal possible growth $T_f(r)$ of
universal entire holomorphic / meromorphic functions  $f:\mathbb{C} \rightarrow \mathbb{CP}^1$, in terms of the Nevanlinna-Shimizu-Ahlfors  characteristic function
\[
T_f(r)
:=
\int_{1}^r
\frac{\dif t}{t}
\int_{\mathbb{D}_t}
f^*\omega\qquad
{\scriptstyle
			(\forall\, r\,\geqslant\, 1)
		},
\]
where $\omega$ is the Fubini-Study form on $\mathbb{CP}^1$ giving unit area, and where $\mathbb{D}_t:=\{z\in \mathbb{C}:|z|<t\}$.

An easy observation is that, a universal holomorphic / meromorphic function $f$ must have infinite area
\[
\int_{\mathbb{D}_t} f^*\omega
\rightarrow
\infty,
\quad
\text{as }
t\rightarrow \infty,
\]
for otherwise by complex analysis~\cite[p.~10, Theorem 1.1.26]{Noguchi-Winkelmann-book}, 
$f$ can be extended holomorphically across the infinity  $\{\infty\}= \mathbb{CP}^1 \setminus \mathbb{C}$ and hence $f: \mathbb{CP}^1 \rightarrow  \mathbb{CP}^1$ is a rational function, 
contradicting the presumed universality.

Thus the growth $T_f$ of a universal holomorphic / meromorphic function $f$ must be sufficiently fast
\[
\liminf_{r\rightarrow +\infty}\,
T_f(r)/\log r
\geqslant
\lim_{r\rightarrow +\infty}
\frac{1}{\log r}
\int_{r_0}^r
\frac{\dif t}{t}
\int_{\mathbb{D}_{r_0}}
f^*\omega
=
\int_{\mathbb{D}_{r_0}}
f^*\omega,
\]
where $\int_{\mathbb{D}_{r_0}}
f^*\omega$ tends to
infinity 
as
$r_0\rightarrow +\infty$.

On the other hand,
given any positive continuous nondecreasing function $\psi: \mathbb{R}_{\geqslant 1}\rightarrow \mathbb{R}_+$ tending to infinity
$
\lim_{r\rightarrow +\infty}
\psi(r)
=
+\infty$, 
it is shown in~\cite{Xie-Chen-Tuan}  that 
\begin{thm}
	\label{chen-hunh-xie thm}
	There exists some
	universal meromorphic function $f: \mathbb{C}\rightarrow \mathbb{CP}^1$ with  slow growth
	\begin{equation}
		\label{admissible-slow-growth}
		T_f(r)
		\leqslant
		\psi(r)
		\cdot
		\log r
		\qquad
		{\scriptstyle
			(\forall\, r\geqslant\, 1).
		}
	\end{equation}
\end{thm}
    
This result  is an optimal answer to the meromorphic part of~\cite[Problem 9.1]{Dinh-Sibony-list}. Indeed, given any universal meromorphic function $f$,  one can construct an auxiliary positive continuous nondecreasing function 
\[
\hat{\psi}(t)
=
\inf_{r> t}\,
\{T_f(r)/\log r\}
\,
:\,
\mathbb{R}_{\geqslant 1}
\longrightarrow
\mathbb{R}_{+}
\]
such that an opposite  growth rate estimate holds automatically 
\[
T_f(r)
\geqslant
\hat{\psi}(r)
\cdot
\log r
\qquad
{\scriptstyle
	(\forall\, r\, \geqslant\, 1),
}
\]
where
$
\hat{\psi}(r)
$
tends to 
infinity as $r\rightarrow \infty$.

However, for the holomorphic part of~\cite[Problem 9.1]{Dinh-Sibony-list},
aiming at the same slow growth rate~\eqref{admissible-slow-growth}, 
the authors had encountered much more difficulties. Generally, for constructing holomorphic functions, the  most successful method shall be about solving $\overline{\partial}$-equations, by either (1)  explicit general Cauchy integral formula~\cite[p.~3, Theorem 1.2.1]{Hormander-SCV-book},
or by (2)  H\"{o}rmander's $L^2$ estimates~\cite[p.~94, Theorem 4.4.2]{Hormander-SCV-book}. Both machinery,
together with several plausible ideas, failed for the authors. 
As a matter of fact, there was also psychological obstacle 
about wondering whether such universal entire functions having slow growth actually exist,
and our faith in {\em Oka principle}~\cite[pp.~368--369]{MR4547869} was in the test.
	
Nevertheless,  a surprise came to the authors, shortly after they successfully designed an {\em Algorithm} producing the desired universal entire functions,  that such result was already sketched in~\cite{DRuis-paper-1}. 

Unlike the ignorance of~\cite{DRuis-paper-1} in the community of complex analysis,  in  hypercyclic operators theory, Du\u{\i}os Ruis’ result is widely known (see~\cite[pp.~109-110, Theorem~4.23]{book-linear-chaos}). Here we rephrase

\begin{thm}[Du\u{\i}os Ruis]
	\label{thm-Duios-Ruis}
	Let $\phi: \,\mathbb{R}_{\geqslant 0}\, \rightarrow\, \mathbb{R}_{> 0}$ be a continuous function growing faster than any polynomial
	\[
	\lim_{r\rightarrow \infty}
	\phi(r)/r^N
	=\infty
	\qquad
	{\scriptstyle
		(\forall\, N\, \geqslant\, 1).
	}
	\] 
	Let $a$ be a nonzero complex number.
	Then there exists an entire function $f$ 
	hypercyclic for $\mathsf{T}_a$ with slow growth
	\begin{equation}
		\label{DR bound}
		|f(z)|\leqslant
		\phi(|z|)
		\qquad
		{\scriptstyle
			(\forall\,z\,\in\, \mathbb{C}).
		}
	\end{equation}
\end{thm}

Let us explain the notation and terminology.
The space $\mathcal{H}(\mathbb{C})$ of entire holomorphic functions has a natural topology given by uniform convergence on compact sets. For a nonzero complex number $a$, the translation operator $\mathsf{T}_a:
\mathcal{H}(\mathbb{C})
\rightarrow
\mathcal{H}(\mathbb{C})$
is defined by
$
f(\bullet)
\mapsto
f(\bullet+a)
$. An element
$f\in \mathcal{H}(\mathbb{C})$ is said to be {\sl hypercyclic} for $\mathsf{T}_a$ if
the orbit
$\{f, \mathsf{T}_a^{(1)}(f),  \mathsf{T}_a^{(2)}(f), \dots\}$ 
is dense in $\mathcal{H}(\mathbb{C})$.

In fact, the original statement in~\cite{DRuis-paper-1}
only concerns  slow growth real analytic functions $\phi$ having positive  coefficients 
shrinking to zero rapidly
\begin{equation}
	\label{admissible analytic function}
	\phi(r)=
	\sum_{i=0}^{\infty}
	A_i\,r^i,
	\qquad
	A_i>0
	\quad
	{\scriptstyle (\forall\, i\,\geqslant\, 0)},
	\qquad
	\lim_{i\rightarrow \infty}
	A_{i+1}/A_i
	=
	0.
\end{equation}
The two statements are indeed equivalent,
see Lemma~\ref{comparison-lemma} below. Hence the holomorphic part of~\cite[Problem 9.1]{Dinh-Sibony-list}  
has a clear answer of the shape~\eqref{admissible-slow-growth} after Du\u{\i}os Ruis' Theorem~\ref{thm-Duios-Ruis}, by using H. Cartan's equivalent version of Nevanlinna's characteristic function 
\begin{equation}
	\label{Cartan order function}
	T_f(r)
	=
	\frac{1}{2\pi}
	\int_{0}^{2\pi}
	\log^+|f(re^{i\theta})|\dif\theta+
	O(1),
\end{equation}
where $O(1)$ stands for some uniformly bounded term~\cite[p.~9]{Noguchi-Winkelmann-book}.

In retrospect, outside complex analysis, Birkhoff's example of universal entire functions 
was phenomenal. It
foreshadowed a new branch of functional analysis / dynamical systems about linear chaos~\cite{MR1685272, book-bayart-matheron, book-linear-chaos}, according to which an infinite dimensional linear system may exhibit  nonlinear and unpredictable patterns.
The typical
example is 

\begin{thm}[Birkhoff~\cite{Birkhoff1929}]\label{Birkhoff 1929 theorem}
    For $a\in \mathbb{C}\setminus \{0\}$, 
    some 
$f\in \mathcal{H}(\mathbb{C})$ is hypercyclic for $\mathsf{T}_a$.
\end{thm}

\begin{ques}
	\label{natural-question}
	Within  transcendental growth rate $\phi$ of the shape~\eqref{admissible analytic function}, can we construct
	universal holomorphic functions $F\in\mathcal{H}(\mathbb{C}^n)$ 
	in several variables
	which are hypercyclic for a given $\mathsf{T}_a$ where $a\in \mathbb{C}^n\setminus \{\mathbf{0}\}$?
\end{ques}

Note that the approach of 
Du\u{\i}os Ruis~\cite{DRuis-paper-1} only works for the
source space $\mathbb{C}$, see the discussion in Subsection~\ref{subsection 2.1}. Our initial goal is thus to answer 
Question~\ref{natural-question} {\em in positive} by adapting our {\em Algorithm}. It turns out that we can prove much more. 

Here is a reminiscent of~\cite[Theorem 8]{amazing-theorem}.

\begin{lem}
	\label{all-in hypercyclic lemma}
	Let $X$ be an Oka manifold.  In the space $\mathsf{Hol}(\mathbb{C}^n, X)$, equipped with compact-open topology, 
	of holomorphic maps from $\mathbb{C}^n$ to $X$,
	if $f$ is a hypercyclic element
	for the translation operator $\mathsf{T}_a$ where $a\in \mathbb{C}^n\setminus \{\mathbf{0}\}$,
	then $f$ is hypercyclic simultaneously for all $\mathsf{T}_b$ where $b\in\mathbb{R}_+\cdot a$.
\end{lem}

 Hence we shall consider the {\em hypercyclic directions} $I$
in the unit sphere $\mathbb{S}^{2n-1}\subset
\mathbb{C}^n$ such that
$F$ is hypercyclic for 
$\mathsf{T}_a$,
where  $a\in I$.

\medskip
\noindent
{\bf Convention.} 
$||\bullet||$ denotes the standard Euclidean norm on any $\mathbb{C}^k$, which shall be clear according to the context. $\mathbb{B}(a, r)$ stands for the ball centered at $a$ having radius $r$ with respect to the norm  $||\bullet||$. When $a$ is the origin, we write $\mathbb{B}(r)$ for shortness of $\mathbb{B}(\mathbf{0}, r)$. We abuse $O(1)$ for  uniformly bounded terms, which will vary. The differential operator $\dif^c$ stands for $\frac{\sqrt{-1}}{4\pi}(\bar{\partial}-\partial)$.

\medskip

\bigskip\noindent {\bf Theorem A.}
{\it
	Given any transcendental  growth function $\phi$ of the shape~\eqref{admissible analytic function},
	for
	countable  directions $\{\theta_i\}_{i\geqslant 1}$ in the unit sphere $\mathbb{S}^{2n-1}\subset
	\mathbb{C}^n$, there exist
	universal  holomorphic maps
	$F: \mathbb{C}^n \rightarrow \mathbb{C}^m$ with slow growth 
	\begin{equation}
		\label{desired growth in several variables}
		||F(z)||\leqslant
		\phi(||z||)
		\qquad
		\qquad
		{\scriptstyle
			(\forall\,z\,\in\, \mathbb{C}^n),
		}
	\end{equation}
	such that the hypercyclic directions of $F$ include $\{\theta_i\}_{i\geqslant 1}$.
	Namely, $F$ are hypercyclic for all  $\mathsf{T}_a$, where  $a$ are in some ray $\mathbb{R}_{+}\cdot \theta_i$. 
}

\bigskip

This is the main result of this paper, which answers Question~\ref{natural-question}.
The proof is by the aforementioned constructive {\em Algorithm}, whose underling idea goes back to an insight of Sibony, see Subsection~\ref{Sibony question}.

\smallskip
One application of Theorem~A is generalizing 
Theorem~\ref{chen-hunh-xie thm}
to higher dimensional universal holomorphic maps
$f: \mathbb{C}^n\rightarrow \mathbb{CP}^m$ ($m\geqslant n\geqslant 1$), keeping in mind that
 the method of~\cite{Xie-Chen-Tuan}  only works for $f: \mathbb{C}\rightarrow \mathbb{CP}^m$ (see Subsection~\ref{expain CHX}).

Recall that a Nevanlinna characteristic function of $f:\mathbb{C}^n \rightarrow \mathbb{CP}^m$ is given by
\[
T_f(r)
:=
\int_{1}^r
\frac{\dif t}{t^{2n-1}}
\int_{\mathbb{B}(t)}
f^*\omega_{\mathsf{FS}}\wedge \alpha^{n-1},
\]
where $\omega_{\mathsf{FS}} $ is the  Fubini-Study form on $\mathbb{CP}^m$ and where
$
\alpha
:=\dif\dif^c||z||^2
$
(cf.~\cite{Noguchi-Winkelmann-book, MR4265173}).

\medskip\noindent {\bf Theorem B.}
{\it
	Let $n\leqslant m$ be two positive integers. 
	Let $\psi: \mathbb{R}_{\geqslant 1}\rightarrow \mathbb{R}_+$
	be a continuous nondecreasing function tending to infinity. Then
	for
	given directions  $\{\theta_i\}_{i\geqslant 1}$ in the unit sphere $\mathbb{S}^{2n-1}\subset
	\mathbb{C}^n$, there exist
	universal entire holomorphic maps
	$f: \mathbb{C}^n \rightarrow \mathbb{CP}^m$ with slow growth 
	\begin{equation}
	\label{admissible-slow-growth cpn}
		T_f(r)
		\leqslant
		\psi(r)
		\cdot
		\log r
		\qquad
		{\scriptstyle
			(\forall\, r\,\geqslant\, 1),
		}
	\end{equation}
	such that  $f$ are hypercyclic for all nontrivial translations 
	$\mathsf{T}_a$,
	where  $a$ are in some ray
	$\mathbb{R}_+\cdot \theta_i$. 
}

\bigskip

Another application of Theorem~A is
determining minimal growth of universal holomorphic maps from $\mathbb{C}^n$ to a complex $m$-tori $\mathbb{T}^m$, i.e., $\mathbb{T}^m=\mathbb{C}^m/\Gamma$ for some lattice $\Gamma$ of full rank $2m$.

\medskip\noindent {\bf Theorem C.}
{\it
	Given any continuous nondecreasing function $\psi: \mathbb{R}_{\geqslant 1}\rightarrow \mathbb{R}_+$
  tending to infinity,
	for
	countably many  directions $\{\theta_i\}_{i\geqslant 1}$ in the unit sphere $\mathbb{S}^{2n-1}\subset
	\mathbb{C}^n$, there exist
	universal  holomorphic maps
	$F: \mathbb{C}^n \rightarrow \mathbb{T}^m$ with slow growth 
	\begin{equation}
		\label{desired growth in several variables tori}
		T_F(r)\leqslant
		r^{\psi(r)}
		\qquad
		\qquad
		{\scriptstyle
			(\forall\,r\,\geqslant\,1),
		}
	\end{equation}
	such that the hypercyclic directions of $F$ contains $\{\theta_i\}_{i\geqslant 1}$. 
}

\medskip

Without loss of generality, in~\eqref{desired growth in several variables tori} we take
	\[
T_F(r)
:=
\int_{1}^r
\frac{\dif t}{t^{2n-1}}
\int_{\mathbb{B}(t)}
F^*\omega_{\mathsf{std}}\wedge \alpha^{n-1},
\]
where $\omega_{\mathsf{std}}:=\dif\dif^c ||z||^2 $ is a  positive  $(1, 1)$ form  induced from the Euclidean space $\mathbb{C}^m$.

\smallskip

Recall a surprising result~\cite{MR3679721}  that some holomorphic functions  $f\in\mathcal{H}(\mathbb{C}^n)$ in several variables are simultaneously hypercyclic for all nontrivial translations $\mathsf{T}_a$ where $a\in \mathbb{C}^n\setminus \{\mathbf{0}\}$.
It is thus natural to ask

\begin{ques}
	\label{natural-question 2}
	Can we strengthen Theorem~A
	by moreover requiring the hypercyclic directions of $F$ to be the whole unit sphere $\mathbb{S}^{2n-1}$?
\end{ques}

Keeping in mind the slogan of modern Oka principle (see Section~\ref{section: reflection}), for the above question,
there seems no obvious topological obstruction at first glance. However,  the answer is  {\em No}, by deep reason in Nevanlinna theory.
Here for simplicity we only analyze the basic case  $F: \mathbb{C}\rightarrow \mathbb{C}$.

\medskip\noindent {\bf Theorem D.}
{\it
	There exist some
	transcendental growth function   $\phi$ of the shape~\eqref{admissible analytic function},
	such that for any entire holomorphic function
	$F\in \mathcal{H}(\mathbb{C})$ with slow growth
	\begin{equation}
		\label{thm 1.8 grow rate}
		|F(z)|\leqslant
		\phi(|z|)
		\qquad
		\qquad
		{\scriptstyle
			(\forall\,z\,\in\, \mathbb{C}),
		}
	\end{equation}
	the hypercyclic directions $I\subset \mathbb{S}^1$ of $F$
	must have Hausdorff dimension zero.
}

\medskip

Since the hypercyclic directions $I$ of a universal holomorphic function $F$ shall be of small size and could contain countable points, it is natural to ask if some $I$ of $F$ can contain uncountable directions. This time the answer is {\em yes}.
 For shortness  we again only consider  the case  $\mathcal{H}(\mathbb{C})$.

\medskip\noindent {\bf Theorem E.}
{\it
	Given any transcendental growth function  $\phi$ of the shape~\eqref{admissible analytic function},
	there exists some
	universal entire holomorphic function
	$F\in \mathcal{H}(\mathbb{C})$ with slow growth
	\[
	|F(z)|\leqslant
	\phi(|z|)
	\qquad
	\qquad
	{\scriptstyle
		(\forall\,z\,\in\, \mathbb{C}),
	}
	\]
	such that  the hypercyclic directions $I\subset \mathbb{S}^1$ of $F$ is uncountable.
}

\medskip

Thus one shall be prudent when using Oka principle, in case there might be certain delicate subtlety 
depending on the nature of the question.

\smallskip

Here is the structure of this paper. In Section~\ref{section: Oka} we
prove a comparison Lemma~\ref{comparison-lemma} to show that
the class of slow growth functions of shape~\eqref{admissible analytic function} in our theorems is optimal, and we
explain our central idea of ``patching'' together holomorphic discs inspired by  Question~\ref{Sibony-last-question} of Sibony.  A relevant speculation of Brunella is also 
recalled.
In Section~\ref{section: algorithm} we present an {\em Algorithm}, which is the engine in our approach for constructing universal entire holomorphic maps with slow growth.
In Section~\ref{section: proofs} we prove Lemma~\ref{all-in hypercyclic lemma} and all our Theorems A, B, C, D, E.
Lastly, in Section~\ref{section: reflection}, we highlight key questions for future research.

\bigskip\noindent
{\bf Acknowledgments.}
We were inspired by an insight  of Sibony~\cite{Sibony2021}.
This paper would never happen  without 
 Forstneri\v{c}'s  excellent online lectures on Oka manifolds theory in AMSS SCV seminar, November 2022.
We thank Zhangchi Chen and Dinh Tuan Huynh for stimulating conversations  and valuable suggestions which improved this manuscript.

\bigskip\noindent
{\bf Funding.}
The second named author is
partially supported by 
National Key R\&D Program of China Grant
No.~2021YFA1003100 and  NSFC Grant No.~12288201.

\bigskip\noindent
{\bf Claim.}  There is no conflict of interest and there is no associated
data in this work.

\section{\bf Background}
\label{section: Oka}

\subsection{About the construction of Du\u{\i}os Ruis}
\label{subsection 2.1}

The original article~\cite{DRuis-paper-1}  of the Theorem~\ref{thm-Duios-Ruis} 
only concerns a special class of analytic functions $\phi$, by the following reason.

\begin{lem}
	\label{comparison-lemma}
	Let $\phi: \,\mathbb{R}_{\geqslant 0}\, \rightarrow\, \mathbb{R}_{> 0}$ be a continuous function growing faster than any polynomials
	\[
	\lim_{r\rightarrow \infty}
	\phi(r)/r^N
	=\infty
	\qquad
	{\scriptstyle
		(\forall\, N\, \geqslant\, 1).
	}
	\]
	Then the function $\phi(r)$ dominates some
	analytic function 
	$
	\sum_{i=0}^{\infty}
	A_i\,r^i
	$
	with positive  coefficients
	$A_i$ shrinking to zero rapidly
	\[
	\lim_{i\rightarrow \infty}
	A_{i+1}/A_i
	=
	0.
	\]
\end{lem}

\begin{proof} 
	Set $A_0$ to be 
	half of the minimum value of 
	the function $\phi(r)$
	on the interval $[0, \infty[$. It is clear that $A_0>0$.
	For $i=0, 1, 2,  \dots$,
	inductively we set
	$A_{i+1}$
	to be the smaller number between
	$\frac{A_i}{i+1}$ and  half of the minimum value of 
	the continuous positive function 
$
	\big(
	\phi(r)-\sum_{j=0}^i A_j r^j
	\big)
	/r^{i+1}
$
	on the interval $]0, \infty[$. It is clear that $A_{i+1}>0$, and that
$
	A_{i+1}/A_i\leqslant \frac{1}{i+1}$,
	$
	\phi(r)
	>
	\sum_{j=0}^{i+1} A_j\, r^j
$
	for all $r\geqslant 0$.
	By taking $i \rightarrow \infty$ we conclude
$
	\phi(r)
	\geqslant
	\sum_{j=0}^{\infty} A_j\, r^j$
	for all $r\geqslant 0$.
\end{proof}

Theorem~\ref{thm-Duios-Ruis} is claimed in the  article~\cite{DRuis-paper-1}, with no   explanations but a few  formulas to indicate the construction, in which   the fundamental theorem of algebra plays an important role:
\begin{equation}
	\label{DR's trick}
	\text{\em every polynomial in }
	\mathbb{C}[z]\,\, \text{\em  factors as products of  one forms.}
\end{equation}
Du\u{\i}os Ruis' universal entire functions with slow growth are constructed by infinite products of certain well-chosen one forms. Therefore, this  method is strictly one dimensional, as~\eqref{DR's trick} fails for
$\mathbb{C}[z_1, \dots, z_n]$ ($n\geqslant 2$).

\subsection{An insight of Sibony}
\label{Sibony question}

After receiving an early manuscript of
\cite{MR4338457}, 
Sibony asked

\begin{ques}[\cite{Sibony2021}]
	\label{Sibony-last-question}
	Show that for $X=\mathbb{CP}^n$, or complex tori, all Ahlfors  currents can be obtained by single entire curve $f: \mathbb{C}\rightarrow X$.	
\end{ques}

 Sibony suggested a hint~\cite{MR1129075}.
 We regret that we had not discussed this problem with him before it was too late.
 
If  only concentric holomorphic discs
$\{f(\mathbb{D}(r))\}_{r>0}$ are allowed to use (as \cite{MR4338457}),
then we do not know the answer.  
If one can use any holomorphic discs $\{f(\mathbb{D}(a, r))\}_{a\in \mathbb{C},  r>0}$, then we can follow
 Sibony's insight by Oka theory. In fact, such phenomenon holds for  compact complex manifolds $Y$ satisfying a weak Runge's approximation property that
 {\em every holomorphic map from a neighborhood of two disjoint closed discs $D_1, D_2\subset D$ into $Y$ can be ``extended'' to a holomorphic map from a neighborhood of a larger closed disc $D\subset \mathbb{C}$ to $Y$ within given small positive error bound with respect to a complete distance $d_Y$ on $Y$.}
 Such manifolds $Y$
 include all Oka-$1$ manifolds, in particular all rationally connected manifolds~\cite{A-F-2023}. 
Modulo certain technical details,  such entire curve can be constructed by ``patching'' together some countable holomorphic discs. The same method can also construct  universal holomorphic maps into Oka manifolds (a result due to Kusakabe~\cite[Theorem 1.4]{Kusababe-first-paper}). For shortness, we illustrate the method by the following 

\medskip\noindent
{\em Proof of Birkhoff's Theorem~\ref{Birkhoff 1929 theorem} using Runge's approximation theorem.}

\smallskip
Fix a countable and dense subfield of $\mathbb{C}$,
say complex rational numbers
\[\mathbb{Q}_{c}:=
\{
z=x+y\sqrt{-1}:
x, y\in \mathbb{Q}
\}.
\]
Since
$\mathbb{Z}_+ \times \mathbb{Z}_+$ is  countable,  there is some bijection
\[
\varphi=(\varphi_1, \varphi_2)\,
:\,
\mathbb{Z}_+
\overset{\sim}{\longrightarrow}
\mathbb{Z}_+ \times \mathbb{Z}_+.
\]
Fix small positive error bounds $\{\epsilon_i\}_{i\geqslant 1}$, say $\epsilon_i=1/2^i$,  such that 
\begin{equation}
    \label{epsilon 0-time}
\lim_{j\rightarrow \infty} \sum_{i\geqslant j}\, \epsilon_i\,
=\,0.
\end{equation}
We arrange all elements in the countable set
$\mathbb{Q}_{c}[z]
	\times
	\mathbb{Q}_+$,
 in some subsequent order 
$
\{
(f_i, r_i)\}_{i\geqslant 1}$.

The punchline, which will be understood later, is to
make another sequence $\{(g_j, \hat{r}_j)\}_{j\geqslant 1}$
so that each $(f_i, r_i)$ repeats infinitely many times in $(g_j, \hat{r}_j)_{j\geqslant 1}$,
e.g.
\begin{equation}
    \label{countable*countable=countable trick}
g_j:=f_{\varphi_1(j)},
\qquad
	\hat{r}_j
	:=
	r_{\varphi_1(j)}.
	\end{equation}
Now we construct a desired entire function $f$ by taking the limits of some convergent holomorphic functions
\[
F_i: \mathbb{D}_{R_i}\rightarrow \mathbb{C},
\quad
R_i \nearrow \infty
\qquad
{\scriptstyle
	(i\,=\,1,\, 2,\, 3,\, \dots)}.
\]

In {\em Step $1$}, we set three key data $F_1:=g_1$, $R_1:=\hat{r}_1$, $c_1:=0$.

Subsequently, in {\em Step $i+1$} for $i=1, 2, 3, \dots$,
we choose a center $c_{i+1}\in \mathbb{Z}_+\cdot a$ far from the origin, so that
the two discs
$\mathbb{D}_{R_i}$ and
$\mathbb{D}(c_{i+1}, \hat{r}_{i+1})$ 
keeps  positive distance, i.e., $|c_{i+1}|>R_i+\hat{r}_{i+1}$.
We define a polynomial $\hat{g}_{i+1}(\bullet):=g_{i+1}(\bullet-c_{i+1})$ on $\mathbb{D}(c_{i+1}, \hat{r}_{i+1})$ 
by translation of $g_{i+1}$. Now
we take a large disc $\mathbb{D}_{R_{i+1}}$ 
having radius 
$
R_{i+1}>|c_{i+1}|+\hat{r}_{i+1}$, so that it 
contains
$\mathbb{D}_{R_i}$ and $\mathbb{D}(c_{i+1}, \hat{r}_{i+1})$ as relatively compact subsets.
By Runge's approximation theorem,
we can ``extend'' $F_i$ and $\hat{g}_{i+1}$ to some polynomial $F_{i+1}:
\mathbb{D}_{R_{i+1}}
\rightarrow
\mathbb{C}$ 
within admissible errors ($d_\mathbb{C}$ stands for the usual Euclidean distance on $\mathbb{C}$)
\begin{align}
	\label{small errors}
	\sup_{z\in \mathbb{D}_{R_i}}\,d_{\mathbb{C}}(F_{i+1}(z), F_i(z))
	\leqslant \epsilon_{i+1},
	\qquad
	\sup_{z\in \mathbb{D}(c_{i+1}, \hat{r}_{i+1})}
	d_{\mathbb{C}}(F_{i+1}(z), \hat{g}_{i+1}(z))
	\leqslant \epsilon_{i+1}.
\end{align}
By~\eqref{epsilon 0-time}, it is clear that
$\{F_i\}_{i\geqslant 1}$
converges to an
entire curve
$f: \mathbb{C}\rightarrow \mathbb{C}$. Moreover,  by using~\eqref{small errors} repeatedly, we see that $f$  approximates each $\hat{g}_{i+1}$ effectively

\begin{align*}
	\sup_{z\in \mathbb{D}(c_{i+1}, \hat{r}_{i+1})}
	d_{\mathbb{C}}(f(z), \hat{g}_{i+1}(z))
	&
	\leqslant
	\sup_{z\in \mathbb{D}(c_{i+1}, \hat{r}_{i+1})}
d_{\mathbb{C}}(F_{i+1}(z), \hat{g}_{i+1}(z))
	+
	\sum_{j\geqslant i+1}
	\sup_{z\in \mathbb{D}_{R_j}}\,d_{\mathbb{C}}(F_{j+1}(z), F_j(z))
	\nonumber
	\\
	&
	\leqslant
	\epsilon_{i+1}+
	\sum_{j\geqslant i+2}
	\epsilon_j
	\,\,
	\rightarrow
	0
	\qquad
	{\scriptstyle
		(\text{as } i\, \rightarrow\, \infty)}
		\qquad
		[\text{by~\eqref{epsilon 0-time}}].
\end{align*}
In other words
\[
\sup_{z\in \mathbb{D}_{\hat{r}_{i+1}}}\,
d_{\mathbb{C}}(f(c_{i+1}+ z), {g}_{i+1}(z))
\rightarrow
0
\qquad
{\scriptstyle
	(\text{as } i\, \rightarrow\, \infty)}.
\]
This finishes the proof, since
$\mathbb{Q}_{c}[z]$ is dense in $\mathcal{H}(\mathbb{C})$ under the compact-open topology.
\qed

\medskip

The same idea also plays a key r\^ole in the proof of Theorem~\ref{chen-hunh-xie thm},
 as we shall now explain.

\subsection{Thinking process  of~\cite{Xie-Chen-Tuan}}
\label{expain CHX}
Without knowing  Du\u{\i}os Ruis' Theorem~\ref{thm-Duios-Ruis},
the  estimate~\eqref{admissible-slow-growth} was  anticipated by Dinh Tuan Huynh
as an optimal answer to~\cite[Problem 9.1]{Dinh-Sibony-list},
based on evidence given by~\cite{MR1768440}, in which some meromorphic functions $g$ with growth $T_g(r)\leqslant\psi(r)(\log r)^2$ were constructed so that any entire function can be approximated by translations of $g$.

In the preceding proof, one sees clearly the idea of ``patching together'' distinct holomorphic discs,
up to small errors, by Runge's approximation property of the ambient manifold $X$. 
For the meromorphic part of~\cite[Problem 9.1]{Dinh-Sibony-list}, where $X=\mathbb{CP}^1$,
such idea works quite  effectively as follows.

Fix a countable and dense subfield of $\mathbb{C}$,
say 
\[\mathbb{Q}_{c}:=
\{
z=x+y\sqrt{-1}:
x, y\in \mathbb{Q}
\}.
\]
First, we can find countable rational functions 
\[
f_i: \mathbb{D}_{r_i}\rightarrow \mathbb{CP}^1
\qquad
{\scriptstyle
	( i\,=\,1,\,2,\, 3,\,\dots;\,\, 0\,<\,r_i\,<\,\infty)}
\] 
such that
any meromophic function can be reached as a limit by elements in $(f_i)_{i\geqslant 1}$.
Indeed, the Cartesian product 
$\mathbb{Q}_{c}(z)\times \mathbb{Q}_+$
of all rational functions  with $\mathbb{Q}_{c}$ coefficients and all positive rational  radii, is countable and sufficient.

Next, we play the same trick as~\eqref{countable*countable=countable trick} to gather a sequence of rational functions
\[
g_j: \mathbb{D}_{\ell_j}
\rightarrow 
X
\qquad
{\scriptstyle
	(j\,\geqslant\, 1)}
\]
by elements of $(f_i)_{i\geqslant 1}$,
such that 
each $f_i$ for $i=1,2,3,\dots$  repeats infinitely many times in $(g_j)_{j\geqslant 1}$.

If all the rational functions $f_i$ decay to zero near the infinity point,
i.e., 
\begin{equation}
	\label{key assumption}
	f_i=p_i/q_i
	\text{\,
		for some polynomials } p_i, q_i\in \mathbb{Q}_{c}[z] \text{ with } \deg p_i < \deg q_i
	\qquad
	{\scriptstyle
		( i\,=\,1,\,2,\, 3,\,\dots)},
\end{equation}
by the argument in the preceding proof,
we can construct a universal meromorphic function 
\begin{equation}
	\label{summation trick}
	F(z)
	:=
	\sum_{j= 1}^{\infty}
	g_j(z-c_j),
\end{equation}
where $c_j\in \mathbb{C}$ are subsequently well chosen, having absolute values increasing sufficiently fast
\[
1
\ll
|c_1|
\ll
|c_2|
\ll
|c_3|
\ll
\cdots.
\]		

Hence the remaining task
for constructing universal meromorphic function is to make sure that~\eqref{key assumption} holds. This can be achieved by approximating and replacing each $f_i$ on $\mathbb{D}_{r_i}$ by countably many rational functions $\{h_{i, j}\}_{j\geqslant 1}$ vanishing at the infinity.
Indeed, if $f_i$ is a polynomial, then the same as~\cite{MR1768440}
we can use Runge's approximation theorem~\cite{Runge-1885}
to find rational functions
of the shape
\[
\hat{h}_{i, j}
=
\sum_{k=1}^{N_{i, j}}\,
\frac{C_{i, j, k}}{z-\xi_{i, j, k}}
\qquad
{\scriptstyle
	(j\,=\,1,\, 2,\, 3,\, \dots; \,\,N_{i, j}\,\in\, \mathbb{Z}_+;\,\, C_{i, j, k}\,\in\, \mathbb{C};\,\,
	\xi_{i, j, k}\,
	\in \,
	\partial \mathbb{D}_{r_i+1})
}
\] 
to approximate $f_i$ on $\mathbb{D}_{r_i}$ uniformly. 
Therefore, since $\mathbb{Q}_c\subset \mathbb{C}$ is dense,
$f_i$ can be approximated  by
\begin{equation}
	\label{desired shape}
	\text{
		rational functions with $\mathbb{Q}_{c}$  coefficients of the shape~\eqref{key assumption}\qquad
		(countable!)}
\end{equation}
uniformly on $\mathbb{D}_{r_i}$.

Yet how to deal with a general rational function $f_i$? A trick comes naturally:
\[
\text{decompose  \quad}
f_i=f_{i, \text{ pole}}+f_{i, \text{ holo}}
\] 
as the sum of the principle pole part $f_{i, \text{ pole}}$ and the remaining polynomial part
$
f_{i, \text{ holo}}.
$
The first part $f_{i, \text{ pole}}$ automatically enjoys the desired decay near $\infty\in \mathbb{CP}^1$,
and obviously can be approximated by~\eqref{desired shape}.
The later part can be handled likewise by using Runge's approximation theorem.
Hence~\eqref{key assumption} can be achieved in practice.

Lastly, by Nevanlinna theory (cf.~\cite{MR4265173, Noguchi-Winkelmann-book}),
the Shimizu-Ahlfors geometric version of $T_F(r)$
is equivalent to Nevanlinna's  characteristic function~\cite{MR1555233}
\[
T_F(r)
= 
N_F(r, \infty)
+
m_F(r, \infty)
+
O(1),
\] 
where $N_F(r, \infty)$
is the counting function
(with logarithmic weight)
about poles of $F$,
and where the proximity function $m_F(r, \infty)$
measures the closeness of $F(\partial \mathbb{D}_r)$ to $\infty\in \mathbb{CP}^1$.
Since every $g_i(z)\approx 0$ for  $|z|\gg 1$, we have
$
m_F(r, \infty)=0
$
for most $r$.
Thus  $T_F(r)\approx N_F(r, \infty)$
can be made to grow as slowly as possible, by throwing subsequently all the centers
\[
1
\ll
|c_1|
\ll
|c_2|
\ll
|c_3|
\ll
\cdots
\]
far away from the origin.

Alternatively, Zhangchi Chen found an  elegant trick to achieve~\eqref{key assumption}:
\[
\text{replacing} \,
f_i\,
\text{\,by countably many rational functions\,}
\,
f_i
\cdot
\Big(\frac{M_{i, j}}{z+M_{i, j}}\Big)^{\deg f_i+1}
\quad
{\scriptstyle
	(j\,=\,1,\, 2,\, 3,\, \dots; \,\,M_{i, j}\,\in \,\mathbb{Q}_c)}
\]
where $r_i\ll |M_{i, 1}| \ll |M_{i, 2}| 
\ll |M_{i, 3}| \ll \cdots\nearrow \infty$.
Thus one avoids using
Runge's theorem, since 
\[
\Big(\frac{M_{i, j}}{z+M_{i, j}}\Big)^{\deg f_i+1}
\rightarrow
1
\qquad
{\scriptstyle
	(\text{as } j\, \rightarrow\, \infty)}
\]
uniformly for all $z\in \mathbb{D}_{r_i}$. 

The same method works for $\mathbb{C}$ to $\mathbb{CP}^m$ (any $m\geqslant 1$) as well, see~\cite{Xie-Chen-Tuan} for details.
However, for constructing universal holomorphic maps from $\mathbb{C}^n$ to $\mathbb{CP}^m$  ($m\geqslant n\geqslant 2$) with slow growth, the approach of~\cite{Xie-Chen-Tuan} does not work, because poles
of meromorphic functions in several variables are no longer  discrete points.

\subsection{A question of Brunella}
In~\cite[p.~200]{MR1703368} Brunella asked: if for any entire curve $f: \mathbb{C}\rightarrow X$
into a compact complex  manifold $X$, one can ``cleverly'' choose some radii sequence $\{r_i\}_{i\geqslant 1}$ such that $\{f(\mathbb{D}_{r_i})\}_{i\geqslant 1}$ produces  some Nevanlinna / Ahlfors current with either trivial singular part or
trivial diffuse part.

Note that for the exotic entire curves  in~\cite{MR4338457}, for generic (in probability sense) choices of 
$\{r_i\}_{i\geqslant 1}\nearrow \infty$,
 Brunella's anticipation is correct. One may wonder what shall be a canonical Nevanlinna / Ahlfors current associates with
 an entire curve $f: \mathbb{C}\rightarrow X$. This problem is intimately caused by {\em flexibility v.s. rigidity } in complex analysis, for which Oka and Kobayashi shall take some responsibility.
 
If counterexamples $f: \mathbb{C}\rightarrow X$ to Brunella's question ever exist, one shall search $X$ among compact Oka manifolds.

\section{\bf Algorithm}
\label{section: algorithm}
\subsection{One variable case}
\label{One variable case}
We use the idea from Subsection~\ref{Sibony question}
to provide a constructive  proof, different to~\cite{DRuis-paper-1} and~\cite[p.~1433]{Chan-Shapiro-91}, of Du\u{\i}os Ruis' Theorem~\ref{thm-Duios-Ruis}.

We enumerate all  elements in the countable set
$
 \mathbb{Q}_{c}[z]
	\times
	\mathbb{Q}_+$
 in some subsequent order
$
\{
(f_i, r_i)\}_{i\geqslant 1}$. 
Fix an auxiliary bijection 
\begin{equation}
    \label{Z2 bijection to Z}
\varphi=(\varphi_1, \varphi_2)\,
:\,
\mathbb{Z}_+
\overset{\sim}{\longrightarrow}
\mathbb{Z}_+ \times \mathbb{Z}_+.
\end{equation}
We
make another sequence $\{(g_j, \hat{r}_j)\}_{j\geqslant 1}$ so that 
\begin{equation}
    \label{repeats infinitely many times}
\text{
each $(f_i, r_i)$ repeats infinitely many times in $(g_j, \hat{r}_j)_{j\geqslant 1}$.}
\end{equation}
For instance, we can take 
\begin{equation}
	\label{set g_j}
	g_j:=f_{\varphi_1(j)},
	\quad
	\hat{r}_j
	:=
	r_{\varphi_1(j)}.
\end{equation}
Such trick was also visible in~\cite[p.~321]{MR4338457} for some exotic construction.

Given analytic growth function $\phi$ of the shape~\eqref{admissible analytic function}, our {\sl Algorithm} will construct  universal entire functions
$
f:=S_0+S_1+S_2+S_3+\cdots
$
inductively, 
where in each {\em Step $k$} for $k=0, 1, 2, \dots$ we add one piece of jigsaw puzzles 
\begin{equation}
    \label{S_k small coefficients}
S_k
:=
\sum_{j=m_k}^{M_k}\,
a_j\,z^j\qquad
\text{
having small  coefficients } |a_j|\leqslant A_j
\end{equation}
where
$
0=m_0=M_0<
m_1\leqslant M_1 <m_{2}\leqslant M_{2}<m_{3}\leqslant M_{3}<\cdots.
$
Thus the desired estimate~\eqref{DR bound} holds immediately
\[
|f(z)|\,
\leqslant\,
\sum_{k=0}^{\infty}
\,|S_k(z)|\,
\leqslant\,
\sum_{k=0}^{\infty}\,
\sum_{j=m_k}^{M_k}\,
|a_j|\,|z|^j\,
\leqslant \,
\sum_{k=0}^{\infty}\,
\sum_{j=m_k}^{M_k}\,
|A_j|\,|z|^j\,
\leqslant \,
\sum_{j=0}^{\infty}\,
|A_j|\,|z|^j\,
=\,
\phi(|z|).
\]

Fix small positive error bounds $\{\epsilon_i\}_{i\geqslant 1}$, say $\epsilon_i:=2^{-i}$, so that
\begin{equation}
\label{sum of epsilons is almost 0}
\lim_{j\rightarrow \infty}\,
\sum_{i\geqslant j}\,
\epsilon_i\,
=\,\
0.
\end{equation}

Set the initial three key data 
\[S_0=\mathbf{0}\in\mathbb{C}[z],\quad 
R_0=2023\in \mathbb{R}_+,\quad 
c_0=0\in \mathbb{C}.
\]
In each Step $k=1, 2, 3, \dots$, we pick a center $c_k\in \mathbb{C}$ far away from the  disc $\mathbb{D}_{R_{k-1}}$ and then cook up
some $S_k$ satisfying:

\begin{itemize}
	\smallskip
	\item[(1)]
	on the disc $\mathbb{D}_{R_{k-1}}$ the polynomial $S_k$ is negligible 
	\begin{equation}
		\label{condition 1}
		|S_k(z)|\leqslant \epsilon_k 
		\qquad
		{\scriptstyle
			(\forall\,z\,\in\, \mathbb{D}_{R_{k-1}})};
	\end{equation}
	
	\smallskip
	\item[(2)]
	the partial sum
	$G_k:=(S_0+S_1+\cdots+S_{k-1})+S_k$ 
	approximates 
	$\mathsf{T}_{-c_k}\,g_k$ on $\mathbb{D}(c_k, \hat{r}_k)$ within small error $\epsilon_k$, i.e.,
	\begin{equation}
		\label{condition 2}
		|G_k(c_k+z)-g_k(z)|\leqslant \epsilon_k
		\qquad
		\qquad
		{\scriptstyle
			(\forall\,z\,\in\, \mathbb{D}_{\hat{r}_k})}.
	\end{equation}
\end{itemize}
We end Step $k$ by choosing a large radius
$R_k>|c_k|+\hat{r}_k$,
and then move on to Step $k+1$.

\bigskip

If the above algorithm runs successfully in all Steps $k=1, 2, 3, \dots$, then  $f$ is universal by virtue of~\eqref{condition 2},~\eqref{repeats infinitely many times},~\eqref{sum of epsilons is almost 0}. We now fulfill the details.

\medskip

The restriction~\eqref{condition 1} can easily be guaranteed by requiring  all coefficients of $S_k$ to be  small enough, say 
\begin{equation}
	\label{small a_j}
	|a_j|\,
	\leqslant\,
	\min\,\{A_j,\,
	\epsilon_k\cdot \big[
	(M_{k}-m_{k}+1)\cdot R_{k-1}^{j}
	\big]^{-1}\}
	\qquad
	{\scriptstyle
		(m_k\,\leqslant\, j\,\leqslant\,M_{k})}.
\end{equation}

However, the proximity goal~\eqref{condition 2} seems doubtful at  first glance. The existence of such $S_k$ is indicated by Runge's  theorem
on polynomial approximations. Indeed, 
if ignoring  the boundedness restriction~\eqref{S_k small coefficients} on the coefficients of $S_k$,
such $S_k$ certainly exists by Runge's  theorem.
By varying $c_k$ far away from the origin continuously, we may expect that certain equivariant $S_k\in \mathbb{C}(c_k)[z]$ exists, and the coefficients are 
decaying to zero as $c_k\rightarrow \infty$. Moreover, we may expect that the degree of $S_k$ with respect to variable $z$ is exactly $\max\{\deg G_{k-1}, \deg g_k\}$.

Now let us clarify the above imagination by explicit computation. The idea is, regarding $c_k$ as a new variable, to construct polynomial $S_k$ with coefficients in the function field $\mathbb{C}(c_k)$ of the shape
\begin{equation}
	\label{a_j shape}
	a_j\,\in\, \frac{1}{c_k}\cdot \mathbb{C}
	\Big[\frac{1}{c_k}
	\Big]
	\qquad
	{\scriptstyle
		(m_k\,\leqslant\, j\,\leqslant\,M_{k})},
\end{equation} 
so that~\eqref{condition 1} automatically holds by later specializing $|c_k|\gg 1$. 
Similarly, we can fix~\eqref{condition 2} by showing that
\begin{equation}
	\label{magic}
	G_k(c_k+z)-g_k(z)
	\text{
		is a polynomial of variable $z$, having coefficients in } \frac{1}{c_k}\cdot \mathbb{C}
	\Big[\frac{1}{c_k}
	\Big].
\end{equation}

Firstly, by gathering binomial expansions, we rewrite
\begin{equation}
	\label{hat-a_k_i}
	(S_0+S_1+\cdots+S_{k-1})(c_k+z)-g_k(z)\,
	=\sum_{i=0}^{\ell_{k}}\, \hat{a}_{k, i}\, z^i,
\end{equation}
where $\ell_k:=\max\{M_{k-1}, \deg g_k\}$,
and where the coefficients
\begin{equation}
	\label{expresion of hat a}
	\hat{a}_{k, i}\,
	=\,
	\sum_{i\leqslant j\leqslant M_{k-1}}\,
	a_j
	\,
	\binom{j}{i}\,c_k^{j-i}\,
	-\,
	\frac{1}{i!}\,
	g^{(i)}_k(0)
	\quad
	\in\, \mathbb{C}[c_k]
	\qquad
	{\scriptstyle
		(i\,=\,0,\,1,\,2,\, \dots,\,\ell_k)}.
\end{equation}

Next, we collect (we will take $m_k>\ell_k$)
\begin{equation}
	\label{h.o.t.}
	S_k(z+c_k)\,
	=\,
	\sum_{j=m_k}^{M_k}\,
	a_j\,(z+c_k)^j\,
	=\,
	\sum_{i=0}^{\ell_k}
	\Big[
	\sum_{j=m_k}^{M_k}\,
	a_j
	\,
	\binom{j}{i}\,c_k^{j-i}
	\Big] 
	\cdot
	z^i
	+(\mathsf{h.o.t.})
\end{equation}
where $(\mathsf{h.o.t})$ are the remaining higher order terms with respect to variable $z$.

Declare
$
m_k:=\ell_{k}+1$ and
$M_k:=2\ell_{k}+1
$
for  solving the system of linear equations 
\begin{equation}
	\label{system of linear equations}
	\sum_{j=m_k}^{M_k}\,
	a_j
	\,
	\binom{j}{i}\,c_k^{j-i}\,
	=\,
	-\,\hat{a}_{k, i}
	\qquad
	{\scriptstyle
		(i\,=\,0,\,1,\,2,\, \dots,\,\ell_k)}
\end{equation}
with the just number $\ell_{k}+1$ of 
unknown variables $\{a_j\}_{m_k\leqslant j\leqslant M_k}$.

\begin{lem}
	The determinant of the above $(\ell_k+1)\times (\ell_k+1)$ coefficient matrix has neat expression
	\begin{equation}
	\label{big determinant}
	\det\,
	\bigg[\binom{j}{i}\,c_k^{j-i}
	\bigg]_{
		\substack{	0\leqslant i\leqslant \ell_k
			\\
			m_k\leqslant j\leqslant M_k
	}}\,
	=\,
	c_k^{(\ell_k+1)^2}.
	\end{equation}
\end{lem}

That the determinant is of  the shape $O(1)\cdot 	c_k^{(\ell_k+1)^2}$ for some nonzero integer $O(1)$ is natural, and is sufficient for our purpose.

\begin{proof}
	Observe that
	\[
	{\scriptstyle
		\big[\binom{j}{i}\,c_k^{j-i}
		\big]_{\substack{	0\leqslant i\leqslant \ell_k
				\\
				m_k\leqslant j\leqslant M_k
		}}\,
		=\,
		\begin{pmatrix}
			c_k^{-0} & 0  & \cdots   & 0 & 0  \\
			0 & c_k^{-1}  & \cdots   & 0 & 0  \\
			\vdots & \vdots  & \ddots   & \vdots & \vdots \\
			0 & 0  & \cdots   & c_k^{-(\ell_k-1)} & 0  \\
			0 & 0  & \cdots   & 0 & c_k^{-\ell_k}
		\end{pmatrix}
		\cdot
		\big[\binom{j}{i}
		\big]_{\substack{	0\leqslant i\leqslant \ell_k
				\\
				m_k\leqslant j\leqslant M_k
		}}
		\cdot 
		\begin{pmatrix}
			c_k^{m_k} & 0  & \cdots   & 0 & 0  \\
			0 & c_k^{m_k+1}  & \cdots   & 0 & 0  \\
			\vdots & \vdots  & \ddots   & \vdots & \vdots \\
			0 & 0  & \cdots   & c_k^{M_k-1} & 0  \\
			0 & 0  & \cdots   & 0 & c_{k}^{M_k}
		\end{pmatrix}
	}.
	\]
	The determinants of the first and the last diagonal matrices contribute $c_k^{m_k^2}$.
	Now	we compute the middle
	$$\det
	\bigg[
	\binom{j}{i}
	\bigg]_{\substack{	0\leqslant i\leqslant \ell_k
			\\
			m_k\leqslant j\leqslant M_k
	}}\,
	=\,
	\det
	\bigg[
	\binom{j}{i}
	\bigg]_{\substack{	0\leqslant i\leqslant \ell_k
			\\
			\ell_k+1\leqslant j\leqslant 2\ell_k+1
	}}\,
	=\,
	\det\begin{pmatrix}
		\binom{\ell_k+1}{0} & \binom{\ell_k+2}{0}  & \cdots   & \binom{2\ell_k}{0} & \binom{2\ell_k+1}{0}  \\
		\binom{\ell_k+1}{1} & \binom{\ell_k+2}{1}  & \cdots   & \binom{2\ell_k}{1} & \binom{2\ell_k+1}{1}  \\
		\vdots & \vdots  & \ddots   & \vdots & \vdots \\
		\binom{\ell_k+1}{\ell_k-1} & \binom{\ell_k+2}{\ell_k-1}  & \cdots   & \binom{2\ell_k}{\ell_k-1} & \binom{2\ell_k+1}{\ell_k-1}  \\
		\binom{\ell_k+1}{\ell_k} & \binom{\ell_k+2}{\ell_k}  & \cdots   & \binom{2\ell_k}{\ell_k} & \binom{2\ell_k+1}{\ell_k}
	\end{pmatrix}.
	$$
	The trick is to use	 the  identity
	$
	\binom{j}{i}-\binom{j-1}{i}=\binom{j-1}{i-1}
	$ for $j\geqslant i\geqslant 1$. 
	Subtracting 
	 the $s$-th column from the $(s+1)$-th column
	subsequently for $s=\ell_k-1, \ell_{k}-2, \dots, 1$, we receive 
	
	\begin{align*}
		\det\bigg[\binom{j}{i}
		\bigg]_{\substack{	0\leqslant i\leqslant \ell_k
				\\
				\ell_k+1\leqslant j\leqslant 2\ell_k+1
		}}
		=\,
		&\det\begin{pmatrix}
			1 & 0  & \cdots   & 0 & 0  \\
			\binom{\ell_k+1}{0} & \binom{\ell_k+1}{0}  & \cdots   & \binom{2\ell_k-1}{0} & \binom{2\ell_k}{0}  \\
			\vdots & \vdots  & \ddots   & \vdots & \vdots \\
			\binom{\ell_k}{\ell_k-2} & \binom{\ell_k+1}{\ell_k-2}  & \cdots   & \binom{2\ell_k-1}{\ell_k-2} & \binom{2\ell_k}{\ell_k-2}  \\
			\binom{\ell_k}{\ell_k-1} & \binom{\ell_k+1}{\ell_k-1}  & \cdots   & \binom{2\ell_k-1}{\ell_k-1} & \binom{2\ell_k}{\ell_k-1}
		\end{pmatrix}\,
		=\,
		\det\bigg[\binom{j}{i}
		\bigg]_{\substack{	0\leqslant i\leqslant \ell_k-1
				\\
				\ell_k+1\leqslant j\leqslant 2\ell_k
		}}.
	\end{align*}
	This is the pattern for induction process. Eventually the outcome is $\det [\binom{\ell_k+1}{0}]=1$.
\end{proof}

\begin{obs}
\label{key observation 3.2}
	In~\eqref{hat-a_k_i},~\eqref{expresion of hat a},~\eqref{system of linear equations} the polynomials $\hat{a}_{k, i}\in \mathbb{C}[c_k]$
	have degrees $\leqslant \ell_k-i$ for $0\leqslant i\leqslant \ell_k$.
	Hence
	by Cramer's rule,
	the solutions $a_j$ of~\eqref{system of linear equations} have shape
	\[
	a_j\,\in\, \frac{1}{c_k^{j-m_k+1}}\cdot \mathbb{C}
	\Big[\frac{1}{c_k}
	\Big]
	\qquad
	{\scriptstyle
		(m_k\,\leqslant\, j\,\leqslant\,M_{k})}.
	\]
	Therefore,  coefficients  of
	$(\mathsf{h.o.t.})$  in~\eqref{h.o.t.}
	are in
	$	\frac{1}{c_k}\cdot \mathbb{C}
	[\frac{1}{c_k}
	]$, whence~\eqref{magic} is satisfied.
\end{obs}

\begin{proof}
By~\eqref{expresion of hat a},
$\hat{a}_{k, i}$ is a polynomial of variable $c_{k}$
with degree $\leqslant M_{k-1}-i$.
Note that $\binom{j}{i}\,c_k^{j-i}$ is of degree $j-i$.
Observe that 
each term in the Laplace expansion of the determinant~\eqref{big determinant} is of the same degree with respect to variable $c_k$.
By Cramer's rule,
we see that
the pole order of $a_j$ with respect to variable $c_k$
is at least 
$ (j-i)-(M_{k-1}-i)=j-M_{k-1}\geqslant j-\ell_k=j-(m_k-1).$
\end{proof}

Summarizing, for every Step $k=1, 2, 3, \dots$, by solving the system of linear equations~\eqref{system of linear equations}, and
by specializing sufficiently large $|c_k|\gg 1$, 
we can obtain a desired  $S_k$
to move on in the Algorithm. In the limit we receive some universal entire functions $f=\sum_{i=0}^{\infty} S_i$
within admissible slow growth $\phi$.

\begin{rmk}
    The trick of choosing $c_k$ far away from the origin is not only sufficient but also necessary, by reason of Nevanlinna theory (see~\eqref{T+O(1)} below), since
    $g_k$ might contain too many zeros in $\mathbb{D}_{\hat{r}_k}$.
\end{rmk}

By improving our algorithm, we can make $f$ to be hypercyclic  simultaneously
for given countable nontrivial translations $\{\mathsf{T}_{a_i}\}_{i\geqslant 1}$. 

Moreover, our Algorithm is flexible
for higher dimensional situation, as we now explain.

\subsection{General case}

Given any transcendental growth function $\phi$ of the shape~\eqref{admissible analytic function},
and given
countably many  directions $\{\theta_i\}_{i\geqslant 1}$ in the unit sphere of
$\mathbb{C}^n$,  we now upgrade  our {\sl Algorithm} to prove Theorem~A, by constructing
universal entire holomorphic maps
$F\in \mathsf{Hol}(\mathbb{C}^n, \mathbb{C}^m)$ with slow growth~\eqref{desired growth in several variables}
such that  $F$ are hypercyclic for all nontrivial translations 
$\mathsf{T}_{\theta_i}$ for $i
\geqslant 1$.

We start with the countable set
\begin{equation}
\label{countable pair}
	\Big(
	\oplus_{j=1}^m\,
	\mathbb{Q}_{c}[z_1, \dots, z_n]
	\Big)\,
	\times\,
	\mathbb{Q}_+,
\end{equation}
in which we enumerate all  elements as
$
\big((f_{i}^{[j]})_{1\leqslant j\leqslant m}, r_i\big)_{i\geqslant 1}$. Using again the trick~\eqref{repeats infinitely many times},
we make another sequence
$
\big((g_{i}^{[j]})_{1\leqslant j\leqslant m}, \hat{r}_i\big)_{i\geqslant 1}$  by declaring 
\begin{equation}
	\label{define g_j}
	g_{i}^{[j]}:=f_{\varphi_1\circ\varphi_1(i)}^{[j]}\,;
	\qquad
	\hat{r}_i
	:=
	r_{\varphi_1\circ\varphi_1(i)}
	\qquad
	{\scriptstyle
		(j\,=\,1,\,2,\, \dots,\,m;\,\, i\,\geqslant \,1)}.
\end{equation}

Our {\sl Algorithm} produces  
universal holomorphic maps \begin{equation}
	\label{cn to cm}
	F=
	(F^{[j]})_{1\leqslant j\leqslant m}: \mathbb{C}^n\longrightarrow \mathbb{C}^m
\end{equation}
by induction,
where each coordinate component
\begin{equation}
	\label{desired F in the main theorem}
	F^{[j]}:=S_0^{[j]}+S_1^{[j]}+S_2^{[j]}+S_3^{[j]}+\cdots
	\qquad
	{\scriptstyle
		(j\,=\,1,\,2,\, \dots,\,m)}.
\end{equation}
In {\em Step $k$} for $k=0, 1, 2, \dots$ we fabricate
\begin{equation}
	\label{S_k CN}
	S_k^{[j]}
	:=
	\sum_{j=m_k}^{M_k}
	\sum_{|I|=j}\,
	a_I^{[j]}\,Z^I
	\qquad
	{\scriptstyle
		(j\,=\,1,\,2,\, \dots,\,m)}.
\end{equation}
with sufficiently small coefficients 
$a_I^{[j]}\in \mathbb{C}$,
so that the desired estimate~\eqref{desired growth in several variables} holds automatically.
Here we use standard notation of 
multi-indices
$I=(i_1, \dots, i_n)\in \mathbb{Z}_{\geqslant 0}^{n}$ with norm $|I|=i_1+\cdots+i_n$,
and 
$Z^I=z_1^{i_1}\cdots z_n^{i_n}$, and likewise we require
$
0=m_0=M_0<
m_1\leqslant M_1 <m_{2}\leqslant M_{2}<\cdots$.

Fix small positive error bounds $\{\epsilon_i\}_{i\geqslant 1}$ as~\eqref{sum of epsilons is almost 0}.
Set the initial  key data 
\[
S_0^{[1]}=\cdots=S_0^{[m]}=\mathbf{0}\,
\in\,
\mathbb{C}[z_1, \dots, z_n], 
\quad
R_0=2023\,\in\, \mathbb{R}_+,
\quad c_0=\mathsf{0}\,\in\, \mathbb{C}^n.
\]
In each Step $k=1, 2, 3, \dots$, we first pick some center 
\begin{equation}
	\label{trick n_k}
	c_k\,\in\, \mathbb{Z}_{+}\cdot \theta_{\varphi_2\circ\varphi_1(k)}
\end{equation} 
far away from  $\mathbb{B}_{R_{k-1}}:=\{
||z||\leqslant R_{k-1}
\}$ in $\mathbb{C}^n$. Next, we fabricate
some $S_k^{[j]}$ of the shape~\eqref{S_k CN}  such that:

\begin{itemize}
	\smallskip
	\item[(1)]
	on the ball $\mathbb{B}_{R_{k-1}}$ the polynomials $S_k^{[j]}$ are negligible 
	\begin{equation}
		\label{condition 1, 2}
		|S_k^{[j]}(z)|\leqslant \epsilon_k 
		\qquad
		{\scriptstyle
			(j\,=\,1,\,2,\, \dots,\,m;\,\,\forall\,z\,\in\, \mathbb{B}_{R_{k-1}})};
	\end{equation}
	
	\smallskip
	\item[(2)]
	the partial sums
	$G_k^{[j]}:=(S_0^{[j]}+S_1^{[j]}+\cdots+S_{k-1}^{[j]})+S_k^{[j]}$ 
	approximates 
	$\mathsf{T}_{-c_k}\,g_k^{[j]}$ on $\mathbb{B}(c_k, \hat{r}_k)$ within small error $\epsilon_k$, i.e.,
	\begin{equation}
		\label{condition 2, 2}
		|G_k^{[j]}(c_k+z)-g_k^{[j]}(z)|\leqslant \epsilon_k
		\qquad
		\qquad
		{\scriptstyle
			(j\,=\,1,\,2,\, \dots,\,m;\,\,\forall\,z\,\in\, \mathbb{B}_{\hat{r}_k})}.
	\end{equation}
\end{itemize}
We end Step $k$ by setting  large radius
$R_k>||c_k||+\hat{r}_k$,
and then move on to Step $k+1$.
 
\bigskip

As before, the restriction~\eqref{condition 1, 2} can easily be satisfied by requiring  all coefficients of $S_k^{[j]}$ to be sufficiently small.
The main trouble comes from~\eqref{condition 2, 2}, which we now handle by the same method.

Complete the unit vector $v_1:=\theta_{\varphi_2\circ \varphi_1(k)}$ to an orthonormal basis 
$
\{v_1, v_2, \dots, v_n\}
$ of $\mathbb{C}^n$.
Define a new linear coordinate system
$(w_1, \dots, w_n)$ for $\mathbb{C}^n$
in which $v_{\ell}$ reads as 
$(0,\dots, 0, 1, 0, \dots, 0)$ where $1$ appears in the ${\ell}$-th position, for ${\ell}=1, \dots, n$.
Rewrite $c_k=n_k\cdot v_1$ where $n_k\in \mathbb{Z}_+$. The idea is,  regarding $n_k$ as a new variable, 
up to changing coordinates linearly
between
$(z_1, \dots, z_n)$ and
$
(w_1, \dots, w_n)
$, 
to construct polynomial 
\begin{equation}
\label{S_k trick}
S_k^{[j]}
=
\sum_{\ell=m_k}^{M_k}\,
a_{\ell}^{[j]}\,w_1^{\ell}
\qquad
{\scriptstyle
	(j\,=\,1,\,2,\, \dots,\,m)}
\end{equation} 
with respect to variable $w_1$ having coefficients
\begin{equation}
\label{a_j shape 2}
a_{\ell}^{[j]}\,\in\, \frac{1}{n_k}\cdot \mathbb{C}[w_2, \dots, w_{n}]
\Big[\frac{1}{n_k}
\Big]
\qquad
{\scriptstyle
	(j\,=\,1,\,2,\, \dots,\,m;\,\,m_k\,\leqslant\, {\ell}\,\leqslant\,M_{k})},
\end{equation} 
so that~\eqref{condition 1, 2} holds by specializing $n_k\gg 1$. 
Similarly, we can fix~\eqref{condition 2, 2} by showing that each polynomial
$G_k^{[j]}(c_k+z)-g_k^{[j]}(z)$,
after changing variables,
is a polynomial of variable $w_1$ 
\begin{equation}
\label{magic 2}
\text{having coefficients in } \frac{1}{n_k}\cdot \mathbb{C}[w_2, \dots, w_{n}]
\Big[\frac{1}{n_k}
\Big].
\end{equation}

Firstly, by changing coordinates and by gathering binomial expansions, we rewrite
\begin{equation}
\label{hat-a_k_i 2}
(S_0^{[j]}+S_1^{[j]}+\cdots+S_{k-1}^{[j]})(c_k+z)-g_k^{[j]}(z)\,
=\,\sum_{i=0}^{\ell_{k}}\, \hat{a}_{k, i}^{[j]}\, w_1^i
\qquad
{\scriptstyle
	(j\,=\,1,\,2,\, \dots,\,m)},
\end{equation}
where $\ell_k:=\max_{1\leqslant j\leqslant m}\{M_{k-1}, \deg g_k^{[j]}\}$,
and where the coefficients 
\begin{equation}
\label{expresion of hat a 2}
\hat{a}_{k, i}^{[j]}\,
\in\,
\mathbb{C}[w_2, \dots, w_{n}]
[n_k]
\qquad
{\scriptstyle
	(j\,=\,1,\,2,\, \dots,\,m;\,\,i\,=\,0,\,1,\,2,\, \dots,\,\ell_k)}.
\end{equation}

Next, by~\eqref{S_k trick} we rewrite
\begin{equation}
\label{h.o.t. 2}
S_k^{[j]}(z+c_k)\,
=\,
\sum_{\ell=m_k}^{M_k}\,
a_{\ell}^{[j]}\,(w_1+n_k)^{\ell}\,
=\,
\sum_{i=0}^{\ell_k}
\Big[
\sum_{\ell=m_k}^{M_k}\,
a_{\ell}^{[j]}
\,
\binom{\ell}{i}\,n_k^{\ell-i}
\Big] 
\cdot
w_1^i
+(\mathsf{h.o.t.})
\qquad
{\scriptstyle
	(j\,=\,1,\,2,\, \dots,\,m)}
\end{equation}
where $(\mathsf{h.o.t})$ are the remaining higher order terms with respect to variable $w_1$.

Declare
$
m_k:=\ell_{k}+1$,
$M_k:=2\ell_{k}+1$.
By solving the linear equations 
\begin{equation}
\label{system of linear equations 2}
\sum_{\ell=m_k}^{M_k}\,
a_{\ell}^{[j]}
\,
\binom{\ell}{i}\,n_k^{\ell-i}\,
=\,
-\,\hat{a}_{k, i}^{[j]}
\qquad
{\scriptstyle
	(j\,=\,1,\,2,\, \dots,\,m;\,\,i\,=\,0,\,1,\,2,\, \dots,\,\ell_k)},
\end{equation}
we obtain the desired $S_k^{[j]}$ as  polynomials of variable $w_1$ with coefficients 
in
\[
\frac{1}{n_k}\cdot \mathbb{C}[w_2, \dots, w_{n}]
\Big[\frac{1}{n_k}
\Big][w_1].
\]
By mimicking the argument of Observation~\ref{key observation 3.2},
we can verify~\eqref{magic 2}. Lastly,
by specializing $n_k\gg 1$ and by changing 
coordinates, 
we recover $S_k^{[j]}$ in variables $z_1, \dots, z_n$ satisfying 
both~\eqref{condition 1, 2} and~\eqref{condition 2, 2}.

\section{\bf Proofs}
\label{section: proofs}
\subsection{Proof of Lemma~\ref{all-in hypercyclic lemma}}
We follow essentially the strategy of~\cite[Theorem 8]{amazing-theorem}.

Our goal is to show that,
for any holomorphic map $g: \overline{\mathbb{B}}(0, r)\rightarrow X$ defined in some open neighborhood of the closed ball $$\overline{\mathbb{B}}(0, r):=\{||z||<r\}\subset \mathbb{C}^n,$$ for any positive error bound $\epsilon>0$,  we can find some positive integer $m$ such that $\mathsf{T}_{b}^{(m)} (f)$ approximates $g$ on $\overline{\mathbb{B}}(0, r)$ up to error $\epsilon$ with respect to a fixed distance $d_X$ on $X$
\begin{equation}
\label{approx g goal}
\max_{z\in \overline{\mathbb{B}}(0,r)}d_X(f(z+m\cdot b), g(z))<\epsilon.
\end{equation}

Of course we must use the condition that $f$ is hypercyclic for $\mathsf{T}_a$. A natural idea is to start with a sequence of  increasing positive integers $\{n_i\}_{i\geqslant 1}$ such that $\{f(\bullet+n_i\cdot a)\}_{i\geqslant 1}$ approximate $g(\bullet)$ near $\overline{\mathbb{B}}(0, r)$. If moreover $\{n_i\}_{i\geqslant 1}$ has a subsequence $\{n_{k_i}\}_{i\geqslant 1}$ such that 
\begin{equation}
\label{arithmetic property}
n_{k_i}\cdot a \approx 0
\quad \text{in}\quad
\mathbb{R}\cdot b/ \mathbb{Z}\cdot b,
\end{equation} 
then by the
continuity of $g$ near $\overline{\mathbb{B}}(0, r)$ 
we are done. However, due to the lack of control on $\{n_i\}_{i\geqslant 1}$, such approach does not work.

The trick in the proof of~\cite[Theorem 8]{amazing-theorem} is that, instead of approximating $g$ by orbit of $f$ under $\mathsf{T}_a$, one constructs certain thoughtful $\hat{g}$ so that an analogue of~\eqref{arithmetic property} holds automatically.

Take a small positive $\delta\ll 1$ such that $g$ is defined 
on $\overline{\mathbb{B}}(0, r+\delta)$ and 
uniformly continuous
\begin{equation}
\label{g uniformly continuous}
d_X(g(z_1), g(z_2))<\epsilon/2023
\qquad
{\scriptstyle
	(\forall \,
	z_1,\, z_2 \,\in 
	\overline{\mathbb{B}}(0, r+\delta)\,\text{with}\,
	||z_1-z_2||\,<\,\delta).
}
\end{equation}
Choose sufficiently many residue classes $\{[c_i]\}_{i=1}^k$ in $\mathbb{R}\cdot b/ \mathbb{Z}\cdot b$ such that
\begin{equation}
\label{condition on ball centers c_k}
\text{	$\{[c_i]\}_{i=1}^k  $ are $\delta$-dense,}
\end{equation} 
i.e. any vector $c\in \mathbb{R}\cdot b$ can be approximated by a representative $c'_i$
of some  class $[c_i]$ within tiny distance $||c-c_i'||<\delta$.
This can be achieved by requiring $k>||b||/2\delta$. 
Next,  for $i=1, 2, \dots, k$, we subsequently 
select some representatives $\hat{c}_i\in \mathbb{R}\cdot b$ of  $[c_i]$
far away from each other so that the closed balls
$\overline{\mathbb{B}}(\hat{c}_i, r+\delta)$ are pairwise disjoint.
Therefore $\cup_{i=1}^k \overline{\mathbb{B}}(\hat{c}_i, r+\delta)$ is polynomial convex~\cite{MR889199}. One can also replace balls with cubes to avoid using~\cite{MR889199}.

Now we define a holomorphic function $g_2$ 
on a neighborhood of $\cup_{i=1}^k \overline{\mathbb{B}}(\hat{c}_i, r+\delta)$ by translations of $g$
\begin{equation}
\label{define g_2}
g_2\restriction_{\overline{\mathbb{B}}(\hat{c}_i, r+\delta)}
(\bullet)
:=
g(\bullet-\hat{c}_i)
\qquad
{\scriptstyle
	(i\,=\,1,\,2,\, \dots,\,k)}.
\end{equation}
Since $X$ is Oka, by the BOPA property (\cite[p.~235]{Forstneric-Oka-book}), we can ``extend'' $g_2$ to a holomorphic map $\hat{g}$ from a neighborhood of large ball $\overline{\mathbb{B}}(0, R)$ with radius
$
R>\max_{1\leqslant i\leqslant k}
\{||\hat{c}_i||\}+r+\delta
$
to $X$ up to small deviation
\begin{equation}
\label{g_2 near g}
d_X(\hat{g}(z), g_2(z))
<
\epsilon/2023
\qquad
{\scriptstyle
	(\forall\, z\,\in\, \cup_{i=1}^k\, \overline{\mathbb{B}}(\hat{c}_i, r+\delta))}.
\end{equation}

Now we take a new sequence of increasing positive integers $\{n_j\}_{j\geqslant 1}$ such that $\{\mathsf{T}_a^{(n_j)}(f)\}_{j\geqslant 1}$ converges to $\hat{g}$ on $\overline{\mathbb{B}}(0, R)$, and that in particular
\begin{equation}
\label{f near g}
\max_{z\in \overline{\mathbb{B}}(0, R)}\,
d_X\,
(f(n_j\cdot a+z), \hat{g}(z))
<
\epsilon/2023
\qquad
{\scriptstyle
	(\forall\, j\,\geqslant\, 1)}.
\end{equation}

For each $j\geqslant 1$,
note that at least one vector in $\{n_j\cdot a+\hat{c}_i\}_{i=1}^k$ is near to $0+\mathbb{Z}\cdot b$ within distance $\delta$,
since the  segment between $- n_j\cdot a-\delta\cdot \frac{b}{||b||}$ and $- n_j\cdot a+\delta\cdot \frac{b}{||b||}$ cannot avoid all vectors in $\{\hat{c}_i+\mathbb{Z}\cdot b\}_{1\leqslant i\leqslant  k}$ by our construction~\eqref{condition on ball centers c_k}. Hence we can rewrite 
\begin{equation}
    \label{basic arithmetic}
n_j\cdot a+\hat{c}_i =m_j\cdot b+\delta’
\end{equation}
for one  $i\in \{1, 2, \dots, k\}$, $m_j\in \mathbb{Z}$ and some residue $||\delta’||<\delta$.

Summarizing~\eqref{g uniformly continuous}, \eqref{define g_2}, \eqref{g_2 near g},
\eqref{f near g},
for any $z$ in $\overline{\mathbb{B}}(0, r)$, we obtain
\begin{align*}
\mathsf{T}_b^{(m_j)}f(z)
&=
f(m_j\cdot b+z)
\\
\text{[from \eqref{basic arithmetic}]}\qquad
&=
f(n_j\cdot a+\hat{c}_i-\delta'+z)
\\
\text{[by~\eqref{f near g}]}
\qquad
&
\approx 
\hat{g}(\hat{c}_i-\delta'+z)
\\
\text{[use~\eqref{g_2 near g}]}
\qquad
&
\approx 
g_2(\hat{c}_i-\delta'+z)
\\
\text{[see~\eqref{define g_2}]}
\qquad
&
=
g(-\delta'+z)
\\
\text{[check~\eqref{g uniformly continuous}]}
\qquad
&
\approx 
g(z).
\end{align*}
Thus the estimate~\eqref{approx g goal} is guaranteed by adding up small differences in the above three ``$\approx$''.
\qed

\subsection{Proof of Theorem~A}
 We now show that
for every $\theta_p$, by our Algorithm, 
any obtained $F$ in~\eqref{cn to cm} is hypercyclic for $\mathsf{T}_{\theta_p}$.
Indeed,  for any pair $\big(f_i=(f_{i}^{[j]})_{1\leqslant j\leqslant m}, r_i\big)$ in~\eqref{countable pair}, for any $\epsilon>0$, we will show that for some integer $n\geqslant  1$ there holds
\begin{equation}
\label{proof requirement of main theorem}
||F(z+n\cdot \theta_p)-f_i(z)||
\leqslant
\epsilon
	\qquad
{\scriptstyle
	(\forall\, z\,\in \, \mathbb{B}(r_i))
}
\end{equation}

By $\{\epsilon_{\ell}\}_{\ell\geqslant 1}$ in~\eqref{sum of epsilons is almost 0},
we can find a large $M$ such that $\sum_{\ell\geqslant M}\epsilon_\ell < \epsilon$.
Recalling~\eqref{define g_j} and~\eqref{trick n_k},
we now select a step number $k>M$ such that 
$
\varphi_1\circ \varphi_1(k)=i$ and $
\varphi_2\circ\varphi_1(k)=p$. 
It is easy to check that $n=n_k$ satisfies the proximity requirement~\eqref{proof requirement of main theorem} by keeping track of the Algorithm.

Lastly, by Lemma~\ref{all-in hypercyclic lemma},
$F$ is hypercyclic for all nontrivial translations 
$\mathsf{T}_a$ where $a\in \mathbb{R}_+\cdot \theta_p$.
\qed

\subsection{Proof of Theorem~B}
The idea is to choose a holomorphic function $F: \mathbb{C}^n\rightarrow \mathbb{C}^{m+1}$ such that:
\begin{itemize}
\smallskip
\item[$\diamond \,1.$]
$F$ has sufficiently slow growth and is hypercyclic for translation operators $\{\mathsf{T}_{\theta_i}\}_{i\geqslant 1}$;

\smallskip
\item[$\diamond \, 2.$]
$F(\mathbb{C}^n)$ avoids the origin
$\mathbf{0}\in \mathbb{C}^{m+1}$.
\end{itemize}
It is clear that $f:=\pi_m\circ F$ satisfies the required hypercyclicity in Theorem~B, where $\pi_m$ is the canonical projection from
$\mathbb{C}^{m+1}\setminus \{\mathbf{0}\}$ to 
$\mathbb{CP}^m$.

The existence of such $F$ is guaranteed by Theorem~A and the following 

\begin{obs}\label{not surjective}
If $m\geqslant n$, then
for any holomorphic map $h:\mathbb{C}^n\rightarrow \mathbb{C}^{m+1}$,
by Sard's theorem, the image $h(\mathbb{C}^n)$ has zero Lebesgue measure in $\mathbb{C}^{m+1}$.
\qed
\end{obs} 

Later, we will take some
universal holomorphic function $\tilde{F}:\mathbb{C}^n\rightarrow \mathbb{C}^{m+1}$ with slow growth, which is  hypercyclic for all $\{\mathsf{T}_{\theta_i}\}_{i\geqslant 1}$,
and choose certain $v\notin \tilde{F}(\mathbb{C}^n)$, so that
$F:=\tilde{F}-v$ satisfies our requirements $\diamond\, 1$ and $\diamond\, 2$ above.
The remaining difficulty  is to make sure~\eqref{admissible-slow-growth cpn}.

\medskip
We need some insight from higher dimensional Nevanlinna theory
for constructing such $F$. 

By homogeneity, the angular form
$
\gamma
:=
\dif^c\log ||z||^2\wedge (\dif \dif^c\log ||z||^2)^{n-1}
$
contributes constant integral
\begin{equation}
\label{constant angular form integral}
\int_{||z||=r} \gamma
=
1	\qquad
{\scriptstyle
	(\forall\, r\,> \, 0)
}
\end{equation}
on each sphere $\{||z||=r\}$. 
By Jensen's formula \cite[Lemma 2.1.33]{Noguchi-Winkelmann-book}, the
characteristic function of $f:=\pi_m\circ F$
can be rephrased (recall $\alpha=\dif \dif^c ||z||^2$)
\begin{equation}\label{high dimension order function 1}
\begin{aligned}
	T_f(r)
	=&
	\int_{1}^r
	\frac{\dif t}{t^{2n-1}}
	\int_{\mathbb{B}(t)}
	\dif \dif^c\log(\sum_{i=0}^{m} |F_i|^2)\wedge \alpha^{n-1}\\
	=&\int_{||z||=r}\log (\sum_{j=0}^m|F_j(z)|^2)^{\frac{1}{2}}\,\gamma(z)-\int_{||z||=1}\log (\sum_{j=0}^m|F_j(z)|^2)^{\frac{1}{2}}\,\gamma(z).
\end{aligned}
\end{equation}

Define $\phi: \mathbb{R}_+\rightarrow \mathbb{R}_+$ to be constantly $1$ on the interval $(0, 1)$ and 
\begin{equation}
\label{define phi(r) by psi}
\phi(r):= r^{\frac{\psi(r)}{2}}
\qquad
{\scriptstyle
	(\forall\,r\,\geqslant\, 1).
}
\end{equation} 
Then $\phi$ is  continuous and nondecreasing.
By Theorem~A, given $\{\theta_i\}_{i\geqslant 1}\subset\mathbb{S}^{2n-1}$, there is some universal holomorphic map $\tilde{F}:\mathbb{C}^n\rightarrow\mathbb{C}^{m+1}$ with slow growth
\begin{equation}\label{estimate1}
||\tilde{F}(z)||\leqslant
\frac{\phi(||z||)}{2^k}
\qquad
\qquad
{\scriptstyle
	(\forall\,z\,\in\, \mathbb{C}^n),
}
\end{equation}
such that $\tilde{F}$ is
hypercyclic for all  translations 
$\{\mathsf{T}_{\theta_i}\}_{i\geqslant 1}$. Here,
the value $k\gg 1$ is to be determined.

By Observation \ref{not surjective}, we can find a vector  $v\in \mathbb{C}^{m+1}\setminus  \tilde{F}(\mathbb{C}^n)$ in the open region
\begin{equation}\label{estimate v}
1-\frac{1}{2^{k-1}}<
||v||<1-\frac{1}{2^{k}}.
\end{equation}
Set $F(z):=\tilde{F}(z)-v$. It is clear that $F$ is
hypercyclic for all $\{\mathsf{T}_{\theta_i}\}_{i\geqslant 1}$. By~\eqref{estimate1},  \eqref{estimate v}, and by $\phi(r)\geqslant 1$ for $r\geqslant 1$, we have
\begin{equation}\label{estimate F}
||F(z)||\leqslant
\frac{\phi(||z||)}{2^k}+1-\frac{1}{2^k}
\leqslant
\phi(||z||)
\qquad
{\scriptstyle
	(\forall\,|z|\,\geqslant\, 1).
}
\end{equation}
In particular
\begin{equation}
\label{upper bound of F on unit sphere}
\sup_{||z||=1} ||F(z)|| 
\leqslant
\phi(1)=1.
\end{equation}
Applying \eqref{estimate1} and \eqref{estimate v} again,  we can estimate the lower bound of $||F||$ on the unit sphere 
\begin{equation}\label{estimate on sphere}
\inf_{||z||=1} ||F(z)||
\geqslant
||v||-
\max_{||z||=1}\,||\tilde{F}(z)||
>
(1 - 
\frac{1}{2^{k-1}})
- \frac{1}{2^k}.
\end{equation}
Thus we can estimate~\eqref{high dimension order function 1} by
\begin{align}
\label{high dimension order function}
T_f(r)
=&\,
\int_{||z||=r}\log (\sum_{j=0}^m|F_j(z)|^2)^{\frac{1}{2}}\gamma(z)-\int_{||z||=1}\log (\sum_{j=0}^m|F_j(z)|^2)^{\frac{1}{2}}\gamma(z)\\
\nonumber
\leqslant&\, \log\phi(r)\cdot\int_{||z||=r} \gamma(z)-\int_{||z||=1}\log ||F(z)||\cdot\gamma(z)
\qquad
\text{[by~\eqref{estimate F}]}
\\
\nonumber
\leqslant&\, \frac{1}{2}\psi(r)\cdot \log r -  \log(1-\frac{1}{2^{k-1}}-\frac{1}{2^k})
\qquad
\text{[by \eqref{constant angular form integral}, \eqref{define phi(r) by psi}, \eqref{upper bound of F on unit sphere},  \eqref{estimate on sphere}]}.
\end{align}
Fix  $\epsilon_0\in (0, 1)$. Choose some large $M_1>0$ such that
\[
- \log\,(1-\frac{1}{2^{k-1}}-\frac{1}{2^k})\leqslant \log\,(1+\epsilon_0)\cdot\frac{\psi(1)}{2}
\qquad
{\scriptstyle
(\forall\,k\,\geqslant\, M_1).
}
\]
Plugging the above estimate into \eqref{high dimension order function}, we get
\begin{equation}
\label{estimate T_f part 1}
T_f(r) \leqslant \psi(r)\cdot \log r
\qquad
{\scriptstyle
(\forall\,r\,\geqslant\, 1\,+\,\epsilon_0).
}
\end{equation}

Next, we prove the remaining part of~\eqref{admissible-slow-growth cpn} for $1<r<1+\epsilon_0<2$ by an alternative method.

The pullback of Fubini-Study form $f^*\omega_{\mathsf{FS}}$ is 
\begin{align*}
\dif\dif^c\log(\sum_{i=0}^{m}||F_i(z)||^2)
=\frac{\sqrt{-1}}{2\pi}\big(    
\frac{\sum_{i,s,l}\frac{\partial F_i}{\partial z_s} dz_s\wedge \frac{\partial \overline{F_i}}{\partial \overline{z_l}}d\overline{z_l}}{||F(z)||^2}
-\frac{(\sum_{i,s}\overline{F_i}\frac{\partial F_i}{\partial z_s}dz_s)\wedge(\sum_{j,l}F_j\frac{\partial\overline{F_j}}{\partial\overline{z_l}}d\overline{z_l})}{||F(z)||^4}
\big).
\end{align*}
We now estimate $F$ and its derivatives on $\{||z||\leqslant 2\}$.

By the same reasoning as~\eqref{upper bound of F on unit sphere} and~\eqref{estimate on sphere}, we can obtain
\[
(1 - 
\frac{1}{2^{k-1}})
- \frac{\phi(2)}{2^k}
\leqslant
\inf_{||z||\leqslant 2} ||F(z)||
\leqslant
\sup_{||z||\leqslant 2} ||F(z)||
\leqslant
\phi(2).
\]
By Cauchy’s integral formula for derivatives, we have
\[
\frac{\partial F_i}{\partial z_s}(z)=\frac{\partial\tilde{F}_i}{\partial z_s}(z)=
\frac{1}{2\pi \sqrt{-1}}
\int_{|\zeta_k-z_k|=1}
\frac{\tilde{F}_i(z_0,\cdots,\zeta_s,\cdots,z_m)}{(z_s-\zeta_s)^2}d\zeta_s
\qquad
{\scriptstyle
(\forall\,||z||\,\leqslant\, 2).
}
\]
Hence by \eqref{estimate1} 
\[
\big|\frac{\partial F_i}{\partial z_s}(z)\big|
\leqslant
\max_{|\zeta_s-z_s|=1}\,|\tilde{F}_i(z_0,\cdots,\zeta_s,\cdots,z_m)|
\leqslant
\frac{\phi(2+1)}{2^k}
\qquad
{\scriptstyle
(\forall\,||z||\,\leqslant\, 2).
}
\]
Hence the $(n, n)$ form $\dif\dif^c\log(\sum_{i=0}^{m}|F_i(z)|^2)\wedge \alpha^{n-1}$ on the ball $\mathbb{B}(2)\subset \mathbb{C}^n$ is bounded from above by $O(1)\cdot \frac{1}{4^k}\cdot \alpha^n$ for some  $O(1)$ depending on $\phi$. Thus for any $r\in [1, 2]$ we have
\begin{align}
\nonumber
T_f(r)
=&
\int_{1}^r
\frac{\dif t}{t^{2n-1}}
\int_{\mathbb{B}(t)}
\dif\dif^c\log(\sum_{i=0}^{m} |F_i|^2)\wedge \alpha^{n-1}
\leqslant
\int_{1}^r
\frac{\dif t}{t^{2n-1}}
\int_{\mathbb{B}(t)}
O(1)\cdot \frac{1}{4^k}\cdot \alpha^{n}
\\
\label{estimate T_f part 2}
=&\, 
O(1) \cdot \frac{1}{4^{k}}
(r^2-1)\,
\leqslant\,
O(1)\cdot \frac{1}{4^k}\cdot
\log r
\qquad[\text{caution: two $O(1)$'s are different}].
\end{align}
Choose $M_2>0$ such that  $O(1)\cdot \frac{1}{4^{M_2}}\leqslant \psi(1)$. 
Setting $k\geqslant  \max\{M_1, M_2\}$, by~\eqref{estimate T_f part 1} and~\eqref{estimate T_f part 2},
we  conclude the proof of~\eqref{admissible-slow-growth cpn}.\qed

\begin{rmk}
When $n>m$, if one  considers universal meromorphic maps or correspondences  between $\mathbb{C}^n$ and
$\mathbb{CP}^m$ (cf. \cite[p.~50]{Noguchi-Winkelmann-book}), then
a likewise estimate~\eqref{admissible-slow-growth cpn} still holds true by much the same construction. However, the more essential problem is about minimal growth of universal holomorphic maps,  and we are interested in whether the shapes depend on $n$ or not.
\end{rmk}

\subsection{Proof of Theorem~C}
Given
countably many  directions $\{\theta_i\}_{i\geqslant 1}$ in the unit sphere $\mathbb{S}^{2n-1}\subset
\mathbb{C}^n$ and a continuous nondecreasing function $\psi: \mathbb{R}_{\geqslant 1}\rightarrow \mathbb{R}_+$
tending to infinity.
Define a continuous positive function
$\phi: \mathbb{R}_{\geqslant 0}\rightarrow \mathbb{R}_+$ by
\[
\phi(r)=
\begin{cases}
	1 \quad& (0\leqslant r <1),\\
r^{\frac{\psi(r)}{2}}\quad &(r\geqslant 1).
\end{cases}
\]
By Theorem~A, we receive  some universal holomorphic map $\tilde{F}:\mathbb{C}^n\rightarrow\mathbb{C}^m$ 
with slow growth
$
||\tilde{F}(\bullet)||\leqslant \phi(||\bullet||)
$,
which is
hypercyclic for translation operators along these directions $\{\theta_i\}_{i\geqslant 1}$.

It is clear that $F=\pi\circ\tilde{F}$ satisfies our desired universality and  hypercyclicity. Now we check  
the growth.  By Jensen's formula~\cite[Lemma 2.1.33]{Noguchi-Winkelmann-book}, for any $r>1$,
we have
$$
	T_F(r)
	=
	\int_{1}^r
	\frac{\dif t}{t^{2n-1}}
	\int_{\mathbb{B}(t)}
	\dif\dif^c||\tilde{F}||^2\wedge \alpha^{n-1}
	=
	\frac{1}{2}\int_{||z||=r}||\tilde{F}(z)||^2\,\gamma(z)
	-\frac{1}{2}\int_{||z||=1}||\tilde{F}(z)||^2\,\gamma(z).
$$
Noting that
$\int_{||z||=r} \gamma
=
1$ and that
$||\tilde{F}(z)||^2\leqslant \phi(||z||)^2
\leqslant r^{\psi(r)}$ for $||z||=r$,
we conclude the proof.
\qed

\begin{rmk}
	The slow growth rate~\eqref{desired growth in several variables tori} is optimal, since  $T_F(\bullet)$ must grow faster than any polynomial due to the universality of $F$.
\end{rmk}

\subsection{Proof of Theorem~D}

Consider the set of hypercyclic angles of $F$ 
\[\hat{I}:=\{[\theta]\in
\mathbb{R}/2\pi \cdot \mathbb{Z}
: F\text{ is hypercyclic for }\, \mathsf{T}_{e^{\sqrt{-1}\theta}}\}.
\]
The angular circle $\mathbb{R}/2\pi \cdot \mathbb{Z}$ has a natural metric induced from the Euclidean norm of $\mathbb{R}$, and our
goal is to show that,
for any $\alpha>0$ and  $\delta>0$, there holds 
\begin{equation}
\label{hausdorff measure = 0}
H_\delta^{\alpha}(\hat{I})=0,
\end{equation} 
where  we use the standard notation
\[
H_\delta^{\alpha}(\hat{I})
=
\inf\big\{
\sum_{i\geqslant 1}\,
(\text{diam}\,U_i)^{\alpha}\,\,
:\,\,
\cup_{i\geqslant 1}\, U_i\supset \hat{I},\, \text{diam}\,U_i<\delta \big\},
\] 
the infimum being taken over all countable covers $\{U_i\}_{i\geqslant 1}$ of $\hat{I}$ having diameters less than $\delta$.

Our strategy of proving~\eqref{hausdorff measure = 0} is by using the slow growth of the zeros-counting function of $F$ 
\begin{equation}
    \label{define n_f}
n_F(t, 0)
:=
\text{cardinality of }\,\{
z\in \mathbb{D}_t
\,:\,
F(z)=0
\}
\text{ counting multiplicities}
\quad 
\qquad
{\scriptstyle
(\forall\, t\,> \,0)}.
\end{equation}
The insight comes from the First Main Theorem in Nevanlinna theory.

Suppose $F\text{ is hypercyclic for }\, \mathsf{T}_{e^{\sqrt{-1}\theta}}$.
It is clear that $F$ cannot be any translation of the identity map $z\mapsto z$.
Hence for any  $\epsilon>0$,
there is some positive integer $r_{\epsilon}$ such that $\mathsf{T}_{e^{\sqrt{-1}\theta}}^{(r_{\epsilon})} (F)$ approximates the identity map on the unit disc
\[
\sup_{|z|\leqslant 1}|F(z+r_{\epsilon}\cdot e^{\sqrt{-1}\theta})-z|<\epsilon.
\]
If $\epsilon$ is sufficiently small, say $\epsilon<1$, then by 
Rouch\'e's theorem, $F$ has exactly one zero in the disc $\mathbb{D}(r_{\epsilon}\cdot e^{\sqrt{-1}\theta},1)$.
By taking an infinite shrinking sequence $\{\epsilon_i\}_{i\geqslant 1}\searrow 0$, we see  the corresponding $\{r_{\epsilon_i}\}_{i\geqslant 1}\nearrow +\infty$. Thus
$F$ must have infinitely many zeros within Euclidean distance  $1$ to the ray $\mathbb{R}_+\cdot e^{\sqrt{-1}\theta}$.

For a fixed universal holomorphic function $F$ in Theorem~C, set
\begin{equation}
\label{cover Ek}
E_k
:=
\{[\theta]\in
\mathbb{R}/2\pi \cdot \mathbb{Z}\,:\,
\exists\, r\in [2^k, 2^{k+1}) \text{ such that  }\,F\,\text{has zeros in}\,\mathbb{D}(r\cdot e^{\sqrt{-1}\theta}, 1)\}
\quad 
{\scriptstyle
	(\forall\, k\,\geqslant \,1)}.
\end{equation}
Then
$
\hat{I}\,
\subset \,
\cap_{\ell\geqslant 1}
\cup_{k\geqslant \ell}\, 
E_k.$

We are going to show that, if $\phi$ grows sufficiently slow, there must hold \begin{equation}
\label{sum is bounded}
\sum_{k\geqslant 1}\,
H^{\alpha}_{\delta}\,
(E_k)\,
<\,
+\infty,
\end{equation}
hence~\eqref{hausdorff measure = 0} follows immediately by the Borel-Cantelli lemma.

Now Nevanlinna theory plays a key r\^ole.
Recall  
Nevanlinna's inequality~\cite[Theorem 1.1.18]{Noguchi-Winkelmann-book} that the growth of $F$ ``controls'' the number of zeros of  $F$ (see~\eqref{define n_f}):
\begin{equation}
\label{T+O(1)}
T_F(r)+O(1)
\geqslant 
N_F(r, 0):=
\int_{t=1}^r\,
n_F(t,0)\,
\frac{\dif t}{t},
\qquad
{\scriptstyle
	(\forall\, r\,>\, 1),
}
\end{equation}
where we constantly abuse the notation $O(1)$ for uniformly bounded terms. For an auxiliary  positive increasing function $\psi: \mathbb{R}_{+}\rightarrow \mathbb{R}_{+}$ tending to infinity $\lim_{r\rightarrow +\infty}\psi(r)=+\infty$, we consider the transcendental growth rate 
\begin{equation}
\label{phi and psi}
\phi(r):=r^{\psi(r)}.
\end{equation}
Then, by Cartan's formula~\eqref{Cartan order function},
the growth condition~\eqref{thm 1.8 grow rate} implies that the left-hand side of~\eqref{T+O(1)}
\[
T_F(r)+O(1)
\leqslant
\psi(r)
\cdot
\log r
+
O(1)
\qquad
{\scriptstyle
(\forall\, r\,>\, 1).
}
\]
Combing the estimates
\[
N_F(r^2,0)\,
\leqslant\,
T_F(r^2)+O(1)\,
\leqslant \,
\psi(r^2)\cdot
\log(r^2)
+
O(1)
\qquad
{\scriptstyle
(\forall\, r\,>\, 1)
}
\]
and
\[
N_F(r^2,0)\,
\geqslant\,
\int_r^{r^2}n_F(r,0)\frac{\dif t}{t}\,
=
\, 
n_F(r,0)\cdot \log r
\qquad
{\scriptstyle
(\forall\, r\,>\, 1)
},
\]
we receive that
\begin{equation}
\label{trick in integral}
n_F(r,0)\,
\leqslant\,
2\cdot
\psi(r^2)
+
o(1)
\qquad
{\scriptstyle
	(\forall\, r\,>\, 1)
}.
\end{equation}
Now we choose $\psi$ in~\eqref{phi and psi} to grow extremely slowly, 
say
\begin{equation}
\label{slow phi}
\psi((2^{k+2})^2)
\leqslant
k
\qquad
{\scriptstyle
	(\forall\, k\,\geqslant\, 1)}.
\end{equation}
Whence
the number of zeros of
$F$  in each annulus 
$
\mathscr{A}_k:=\{
2^k-1
< |z|<2^{k+1}+1\}
$
(for $k\geqslant 1$)
is at most~(see~\eqref{trick in integral}, ~\eqref{slow phi})
$$
n_F(2^{k+1}+1,0)
-
n_F(2^{k}-1,0)
\leqslant
n_F(2^{k+2},0)\,
\leqslant\,
2\cdot
\psi((2^{k+2})^2)
+o(1)
\leqslant 2k+o(1).$$
Hence by~\eqref{cover Ek} and by  plane geometry, $E_k$ can be covered by intervals of the shape 
$
[\theta_{j}-\tfrac{\pi}{r_j}, \theta_{j}+\tfrac{\pi}{r_j}]$
(mod $2\pi\cdot \mathbb{Z}$),
where  
$
r_j\cdot e^{\sqrt{-1}\theta_j}\text{ is a zero of $F$ in } \mathscr{A}_k$. 
Thus
\[
H_\delta^{\alpha}\,
(E_k)\,
\leqslant\,
\big(2k+o(1)\big)
\cdot
H_\delta^{\alpha}(
[-\tfrac{\pi}{2^k-1}, \tfrac{\pi}{2^k-1}]
)
\leqslant\,
\big(2k+o(1)\big)
\cdot
(\tfrac{2\pi}{2^{k}-1})^{\alpha}
\qquad
{\scriptstyle
(\forall\, k\,\gg\, 1,\,\text{s.t.}\,\tfrac{2\pi}{2^{k}-1}\,<\,\delta).
}
\]
Therefore
\[
\sum_{k\geqslant 1}\,
H^{\alpha}_{\delta}(E_{k})<+\infty
\] 
is guaranteed for any $\alpha>0$.

Lastly, we use Lemma~\ref{comparison-lemma} to replace $\phi$ in~\eqref{phi and psi} by an inferior real analytic function of the shape~\eqref{admissible analytic function}. This concludes the proof.
\qed

\subsection{Proof of Theorem~E}
The insight is that certain Cantor set in the interval $[0, 1]$ has zero Lebesgue measure but is uncountable.
We now modify our algorithm  to construct some direction set $I$ likewise.

Notice that each positive integer $n$ has a finite length $l(n)$ in binary expansion
$$
n=\sum_{k=0}^{l(n)}a_k\, 2^k,\quad a_0,a_1,a_2,...,a_{l(n)}\in\{0,1\},
\quad a_{l(n)}=1.
$$
We  reset the sequence $\{(g_j, \hat{r}_j)\}_{j\geqslant 1}$ in~\eqref{set g_j} as
$
g_j:=f_{\varphi_1(l(j))}$,
$
\hat{r}_j
:=
r_{\varphi_1(l(j))}.
$
Thus each $(f_i, r_i)$ for $i=1, 2, 3, \dots$ still repeats infinitely many times in $(g_j, \hat{r}_j)_{j\geqslant 1}$, while
\begin{equation}\label{new pair}
g_{2^k}
=g_{2^k+1}
=\cdots=g_{2^{k+1}-1}=f_{\varphi_1(l(2^k))},
\quad
\hat{r}_{2^k}=\hat{r}_{2^k+1}=\cdots=\hat{r}_{2^{k+1}-1}
\qquad
{\scriptstyle
	(\forall\, k\,\geqslant \,0).
}
\end{equation}

The second place we adjust is the choice of $c_k=r_k\cdot e^{\sqrt{-1}\theta_k}$ in polar coordinates $r_k\in\mathbb{Z}_+$, $\theta_k\in \mathbb{R}/{2\pi\cdot \mathbb{Z}}$. Since each $G_k$ obtained by our Algorithm is continuous and satisfies~\eqref{condition 2}, we can retain the modulus $r_k$ and slightly perturb the argument $\theta_k$ of $c_k$ within a small angle $0<\delta_k\ll 1$ so that
a similar estimate 
\begin{equation}\label{new estimate}
|G_k(r_k\cdot e^{\sqrt{-1}\eta}+z)-g_k(z)|\leqslant 2\epsilon_k
\qquad
\qquad
{\scriptstyle
	(\forall\,z\,\in\, \mathbb{D}_{\hat{r}_k};
	\,\,\forall\,\eta\,\in\,(\theta_k-\delta_k,\,\theta_k+\delta_k)
	\,\text{mod}\, 2\pi\cdot \mathbb{Z})}
\end{equation}
still holds true. By our Algorithm, we can subsequently choose $r_k\in \mathbb{R}_+$, $\theta_k\in \mathbb{R}/{2\pi\cdot \mathbb{Z}}$ and $0<\delta_k\ll 1$ for $k=1, 2, 3, \dots$ with the additional requirements that
$(\theta_{2k}-\delta_{2k},\theta_{2k}+\delta_{2k})$ and $(\theta_{2k+1}-\delta_{2k+1},\theta_{2k+1}+\delta_{2k+1})$ are disjoint subsets of $(\theta_k-\delta_k,\theta_k+\delta_k)$
and that
$1\gg \delta_1 \gg
\delta_2 \gg
\delta_3 \gg
\cdots.$

Gathering all possible angles
\[
E_k=\bigcup_{i=2^k}^{2^{k+1}-1}(\theta_i-\delta_i,\theta_i+\delta_i)
\qquad
{\scriptstyle
(\forall\, k\,\geqslant \,0)}
\]
obtained in adjacent steps,
we see a  fast decreasing pattern
$
E_0\supset E_1\supset E_2\supset\cdots.
$
Thus the Cantor-like set $\cap_{k=0}^{+\infty}\,E_k$ is  uncountable.

On the other hand, by~\eqref{new pair},~\eqref{new estimate} and by the arguments in Subsection~\ref{One variable case}, we can show that all arguments $\theta\in \cap_{k=0}^{+\infty}\,E_k$ produce $a=e^{\sqrt{-1}\cdot \theta}$ such that that $F$ is hypercyclic for $\mathsf{T}_a$.

It suffices to check that for any pair $(f_i,r_i)$ in Subsection~\ref{One variable case} and any $\epsilon>0$, there is some large integer $n\gg 1$ such that
\begin{equation}\label{approximate uncountable}
	||F(z+n\cdot \theta)-f_i(z)|| \leqslant \epsilon
\end{equation}
uniformly for $z\in\mathbb{B}(r_i)$.
Using the sequence $\{\epsilon_{\ell}\}_{\ell\geqslant 1}$ defined in~\eqref{sum of epsilons is almost 0}, we can find a large integer $M$ such that $\sum_{\ell\geqslant M}\epsilon_\ell < \epsilon$. Then, recalling~\eqref{new pair} and~\eqref{trick n_k}, we select a step number $k>M$ such that
$
\varphi_1\circ l(k)=i$ and $\theta\in (\theta_k-\delta_k,\theta_k+\delta_k)$.
It is easy to check that $n=n_k$ satisfies the proximity requirement~\eqref{approximate uncountable} by chasing the algorithm and by noticing~\eqref{new estimate}.

Lastly, by  Lemma~\ref{all-in hypercyclic lemma}, we conclude the proof.
\qed

\section{\bf Reflections}
\label{section: reflection}
Constructing interesting holomorphic objects is one central theme in complex geometry. 
Let us recall the  modern Oka principle~\cite[p.~369]{MR4547869}:

\smallskip
{\em
Analytic problems on Stein manifolds which can be formulated in terms of maps to Oka
manifolds, or liftings with respect to Oka maps, have only topological obstructions.}

\smallskip
In light of Nevanlinna theory
we shall test the above slogan by seeking
 interesting   holomorphic maps into Oka manifolds $Y$ with
 slow growth. See e.g.~\cite{MR3078344}  in this line of thought.
In particular, we are interested in minimal growths of
universal holomorphic maps from various source spaces into $Y$.
In general, such questions are very difficult, because  certain effective version of Runge's approximation theorem ought to be required.
For instance, from $\mathbb{C}^n$ to $\mathbb{C}^m$, it is our Algorithm.
Nevertheless, we would like to ask  

\begin{ques}
Is it true that for any compact Oka manifold $Y$, there exist some universal entire curves
$f: \mathbb{C}\rightarrow Y$ having slow growth no more than the shape~\eqref{desired growth in several variables tori}? 
\end{ques}

\begin{ques}
If the answer to the above question is yes, 
is the same statement holds true for any  compact Oka-$1$ manifold? 
\end{ques}

Recall Kusakabe's result~\cite[Theorem 1.4]{Kusababe-first-paper} that an Oka manifold $Y$ admits universal holomorphic maps $f: X\rightarrow Y$ from various source spaces $X$ with mild symmetry.

\begin{ques}
What can we say about  minimal growth of  $f$ from other source space $X\neq \mathbb{C}$?
Here if $Y$ is not compact, 
we shall find
 appropriate metrics on $Y$  for defining Nevanlinna characteristic functions $T_f$.
\end{ques}

The above  questions seek quantitative invariants of Oka (resp. Oka-$1$) manifolds. They
could also provide new insight for certain classification problems, e.g., whether any K3 surface or Fano manifold is Oka or not. 
Moreover, the techniques developed along this line shall shed light for  constructing  entire and rational  curves on Fano manifolds {\em via analytic methods}, which has been a long standing challenge  in complex geometry.

\bigskip
\begin{center}
\bibliographystyle{alpha}
\bibliography{article}
\end{center}

\end{document}